\documentclass[11pt,twoside]{article}

\usepackage{mathrsfs,amsfonts,amsmath,amssymb}
\usepackage{latexsym}
\newtheorem{satz}{Theorem}
\newtheorem{proposition}[satz]{Proposition}
\newtheorem{theorem}[satz]{Theorem}
\newtheorem{lemma}[satz]{Lemma}
\newtheorem{definition}[satz]{Definition}
\newtheorem{corollary}[satz]{Corollary}
\newtheorem{remark}[satz]{Remark}

\def\no{\noindent}

\def\eps{\varepsilon}
\def\_phi{\varphi}

\def\a{\alpha}
\def\d{\delta}

\def\v{\vec}
\def\F{{\mathbb F}}

\def\m{\times}

\def\ov{\overline}

\def\C{{\mathbb C}}
\def\R{{\mathbb R}}
\def\E{\mathsf {E}}
\def\T{{\mathbb T}}

\def\Z_N{{\mathbb Z}_N}
\def\Z{{\mathbb Z}}
\def\N{{\mathbb N}}

\def\f{{\mathbb F}}
\def\Gr{{\mathbf G}}

\def\l{\left}
\def\r{\right}

\def\oT{{\rm T}}

\def\supp{{\rm supp\,}}
\def\tr{{\rm tr\,}}

\def\G{\Gamma}
\def\FF{\widehat}
\def\c{\circ}

\def\Cf{{\mathcal C}}

\def\gs{\geqslant}
\def\T{\mathsf {T}}

\author{Shkredov I.D.}
\title{ Some new inequalities in additive combinatorics
\footnote{
This work was supported by grant RFFI NN
06-01-00383, 11-01-00759, Russian Government project 11.G34.31.0053,
Federal Program "Scientific and scientific--pedagogical staff of innovative Russia" 2009--2013
and
grant Leading Scientific Schools N 2519.2012.1.}
}
\date{}
\begin{document}
\maketitle

\begin{center}
 Annotation.
\end{center}

{\it \small
    In the paper
    we find
    new inequalities
    involving the
    intersections $A\cap (A-x)$
    of shifts of some subset $A$ from an abelian group.
    We apply the inequalities to obtain new upper bounds for the additive energy
    of multiplicative subgroups and convex sets and also a series another results on the connection of the additive
    energy and so--called higher moments of convolutions.
    Besides
    we prove new theorems on multiplicative subgroups
    concerning
    lower bounds for its doubling constants,
    sharp lower bound for the cardinality of sumset of a multiplicative subgroup and its subprogression
    and another results.
}
\\

\section{Introduction}
\label{sec:introduction}

There are two general ideas in additive combinatorics which are opposite to each other in some sense.
The first one is the following.
Let $\Gr = (\Gr,+)$ be a group and $A$ be an arbitrary subset of $\Gr$.
If we want to obtain an information about the additive structure of our set $A$
then it is useful to consider "more smooth"\, and larger objects like sumsets $A+A$, $A-A$, $A+A+A$
and so on (see \cite{tv}).
Finding good additive structure in sumsets can be used to get useful information about the original set $A$.
The second
idea is to consider smaller objects like $A\cap (A-x)$ and its generalizations
to obtain some required properties of $A$ again.
The
latter
approach is presented brightly in papers \cite{Gow_4}, \cite{Gow_m}
and once more time, recently, in \cite{Schoen_Freiman}.
In the article we concentrate on the last method
and find new connections between the sets $A_x := A\cap (A-x)$ and the original set $A$.

The paper based on
so--called eigenvalues method
(see papers \cite{s} and \cite{ss_E_k})
as well as
Proposition \ref{p:weight}.
To obtain the proposition
we develop the method from
\cite{ss,ss2,sv}
choosing some weight optimally and use a simple fact that
$x$ belongs to $A-A_s$ iff $s$ belongs to $A-A_x$.
The eigenvalues method can be represented, very roughly speaking, as follows.
The important role in additive combinatorics plays so--called the {\it additive energy} of a set $A$,
that is the sum $\E(A) := \sum_{x} |A_x|^2$.
We rewrite the sum as the action of a matrix
$$
    \E(A)
        = \sum_{x,y} (\chi_A \c \chi_A) (x-y) \chi_A (x) \chi_A (y) = \langle \T \chi_A, \chi_A\rangle \,,
$$
where $\chi_A$ is the characteristic function of $A$,
by $\chi_A \c \chi_A$ we denote the convolution of $\chi_A$ (see the definition in the section \ref{sec:definitions})
and the square matrix $\T$ is $\T_{x,y} := (\chi_A \c \chi_A) (x-y)$, $x,y\in A$.
Studying the eigenvalues and the eigenfunctions of $\T$, we obtain the information about
the initial object $\E(A)$.
Another idea here is an attempt to use "local"\, analysis on $A$
in contrast to Fourier transformation method which is defined on the whole group.
Our
approach
is especially useful in the situation when
$A$ coincide with a multiplicative subgroup of the finite field.
The reason is that we know all eigenvalues as well as eigenfunctions in the case.

The simplest consequences of the results are
unusual
inequalities
\begin{equation}\label{f:heart}
    \sum_x \frac{|A_x|^2}{|A\pm A_x|} \le |A|^{-2} \sum_x |A_x|^3 \,,
\end{equation}
and
\begin{equation}\label{f:heart_2}
    \sum_{x,y,z \in A} |A_{x-y}| |A_{x-z}| |A_{y-z}| \ge |A|^{-3} ( \sum_x |A_x|^2 )^3 \,.
\end{equation}
These formulas combining with
another ingredient,
so--called
Katz--Koester inequality (see \cite{kk})
\begin{equation}\label{f:kk_introduction}
    |(A+A)\cap (A+A-x)|\gs |A+(A\cap (A-x))|
\end{equation}
allow us to prove a series of applications (see sections \ref{sec:applications}, \ref{sec:applications2}).
Here we give just two of them.

First of all recall the
previous
results.
In \cite{H-K} (see also \cite{K_Tula}) the following
theorem
was obtained.

\begin{theorem}
    Let $p$ be a prime number, and $\Gamma \subseteq (\Z/p\Z) \setminus \{ 0 \}$
    be a multiplicative subgroup, $|\Gamma| = O(p^{2/3})$.
    Then
    $$
        \E(\G) = O (|\Gamma|^{5/2}).
    $$
\label{t:Konyagin_energy}
\end{theorem}

Recall that a set $A\subseteq \R$ is called {\it convex} if it is the image of a convex map.
In paper \cite{ikrt} a result similar to Theorem \ref{t:Konyagin_energy} for convex sets
was
proved.

\begin{theorem}
    Let $A\subseteq \R$ be a convex set.
    Then
    $$
        \E (A) = O (|A|^{5/2}).
    $$
\label{t:ikrt_energy}
\end{theorem}


It is known that statistical properties of multiplicative subgroups and convex sets are quite similar
(see, e.g. section \ref{sec:preliminaries}).
In particular, both objects have very small characteristic $\E_3$, that is the sum $\sum_x |A_x|^3$.
The last situation exactly the case when our method works very well.
Besides
we exploit
some additional irregularity properties of multiplicative subgroups and convex sets
(see e.g. general Theorem \ref{t:energy_gen} of section \ref{sec:applications2}).
Using our approach we prove that the constant $5/2$ in Theorems \ref{t:Konyagin_energy}, \ref{t:ikrt_energy}
can be replaced by $5/2-\eps_0$, where $\eps_0>0$ is an absolute constant.
The question was asked to the author by Sergey Konyagin.
Certainly, the result implies that $|\G\pm \G| \ge |\G|^{3/2+\eps_0}$ and $|A\pm A| \ge |A|^{3/2+\eps_0}$
for any subgroup and a convex set, correspondingly.
Nevertheless another methods from papers \cite{Li,ss,ss2,sv} and also Corollary 29 of section \ref{sec:applications}
give better bounds for the doubling constant here.
Further applications of inequalities (\ref{f:heart}), (\ref{f:heart_2})
can be found in sections \ref{sec:applications}, \ref{sec:applications2}.

The paper is organized as follows.
We start with definitions and notations used in the article.
The instruments from
section \ref{sec:weighted_KK} concern to sumsets estimates, basically.
Here we give our weighted version of Katz--Koester trick.
On the other hand the tools from the next section \ref{sec:eigenvalues} will be applied to
obtain new bounds for the additive energy.
The main principle here is the following.
Basically,
an upper bound for $\E_3(A)$ does not imply something nontrivial
concerning the additive energy (up to H\"{o}lder inequality, of course)
but if we know a little bit more about irregularity of $A$
then it is possible to obtain a nontrivial upper bound for $\E(A)$.
The rigorous statements are contained in sections \ref{sec:applications} and \ref{sec:applications2}.
Besides inequalities (\ref{f:heart}),  (\ref{f:heart_2}) and Katz--Koester trick
we extensively use the methods from \cite{ss_E_k} in
our proof.


The author is grateful to Sergey Konyagin, Misha Rudnev and Igor Shparlinski
for useful discussions
and, especially, Tomasz Schoen for very useful and fruitful explanations and  discussions.
Also I
acknowledge
Institute IITP RAS for providing me with  excellent working conditions.

\section{Definitions}
\label{sec:definitions}

Let $\Gr$ be an abelian group.
If $\Gr$ is finite then denote by $N$ the cardinality of $\Gr$.
It is well--known~\cite{Rudin_book} that the dual group $\FF{\Gr}$ is isomorphic to $\Gr$ in the case.
Let $f$ be a function from $\Gr$ to $\mathbb{C}.$  We denote the Fourier transform of $f$ by~$\FF{f},$
\begin{equation}\label{F:Fourier}
  \FF{f}(\xi) =  \sum_{x \in \Gr} f(x) e( -\xi \cdot x) \,,
\end{equation}
where $e(x) = e^{2\pi i x}$.
We rely on the following basic identities
\begin{equation}\label{F_Par}
    \sum_{x\in \Gr} |f(x)|^2
        =
            \frac{1}{N} \sum_{\xi \in \FF{\Gr}} \big|\widehat{f} (\xi)\big|^2 \,.
\end{equation}
\begin{equation}\label{svertka}
    \sum_{y\in \Gr} \Big|\sum_{x\in \Gr} f(x) g(y-x) \Big|^2
        = \frac{1}{N} \sum_{\xi \in \FF{\Gr}} \big|\widehat{f} (\xi)\big|^2 \big|\widehat{g} (\xi)\big|^2 \,.
\end{equation}
and
\begin{equation}\label{f:inverse}
    f(x) = \frac{1}{N} \sum_{\xi \in \FF{\Gr}} \FF{f}(\xi) e(\xi \cdot x) \,.
\end{equation}
If
$$
    (f*g) (x) := \sum_{y\in \Gr} f(y) g(x-y) \quad \mbox{ and } \quad (f\circ g) (x) := \sum_{y\in \Gr} f(y) g(y+x)
$$
 then
\begin{equation}\label{f:F_svertka}
    \FF{f*g} = \FF{f} \FF{g} \quad \mbox{ and } \quad \FF{f \circ g} = \FF{f}^c \FF{g} = \ov{\FF{\ov{f}}} \FF{g} \,,
\end{equation}
where for a function $f:\Gr \to \mathbb{C}$ we put $f^c (x):= f(-x)$.
 Clearly,  $(f*g) (x) = (g*f) (x)$ and $(f\c g)(x) = (g \c f) (-x)$, $x\in \Gr$.
 The $k$--fold convolution, $k\in \N$  we denote by $*_k$,
 so $*_k := *(*_{k-1})$.
 It is unimportant but write for definiteness
 $$(f \c_k f) (x) := \sum_{y_1,\dots,y_k} f(y_1) \dots f(y_k) f(x+y_1+\dots+y_k)\,.$$

We use in the paper  the same letter to denote a set
$S\subseteq \Gr$ and its characteristic function $S:\Gr\rightarrow \{0,1\}.$
Write $\E(A,B)$ for {\it additive energy} of two sets $A,B \subseteq \Gr$
(see e.g. \cite{tv}), that is
$$
    \E(A,B) = |\{ a_1+b_1 = a_2+b_2 ~:~ a_1,a_2 \in A,\, b_1,b_2 \in B \}| \,.
$$
If $A=B$ we simply write $\E(A)$ instead of $\E(A,A).$
Clearly,
\begin{equation}\label{f:energy_convolution}
    \E(A,B) = \sum_x (A*B) (x)^2 = \sum_x (A \circ B) (x)^2 = \sum_x (A \circ A) (x) (B \circ B) (x)
    \,.
\end{equation}
and by (\ref{svertka}),
\begin{equation}\label{f:energy_Fourier}
    \E(A,B) = \frac{1}{N} \sum_{\xi} |\FF{A} (\xi)|^2 |\FF{B} (\xi)|^2 \,.
\end{equation}
Let
$$
   \T_k (A) := | \{ a_1 + \dots + a_k = a'_1 + \dots + a'_k  ~:~ a_1, \dots, a_k, a'_1,\dots,a'_k \in A \} | \,.
$$
Let also
$$
    \sigma_k (A) := (A*_k A)(0)=| \{ a_1 + \dots + a_k = 0 ~:~ a_1, \dots, a_k \in A \} | \,.
$$
Notice that for a symmetric set $A$ that is $A=-A$ one has $\sigma_2
(A) = |A|$ and $\sigma_{2k} (A) = \T_k (A)$.

 For a sequence $s=(s_1,\dots, s_{k-1})$ put
$A^B_s = B \cap (A-s_1)\dots \cap (A-s_{k-1}).$
If $B=A$ then write $A_s$ for $A^A_s$.
Let
\begin{equation}\label{f:E_k_preliminalies}
    \E_k(A)=\sum_{x\in \Gr} (A\c A)(x)^k = \sum_{s_1,\dots,s_{k-1} \in \Gr} |A_s|^2
\end{equation}
and
\begin{equation}\label{f:E_k_preliminalies_B}
\E_k(A,B)=\sum_{x\in \Gr} (A\c A)(x) (B\c B)(x)^{k-1} = \sum_{s_1,\dots,s_{k-1} \in \Gr} |B^A_s|^2
\end{equation}
be the higher energies of $A$ and $B$.
The second formulas in (\ref{f:E_k_preliminalies}), (\ref{f:E_k_preliminalies_B})
can be considered as the definitions of $\E_k(A)$, $\E_k(A,B)$ for non integer $k$, $k\ge 1$.

Clearly,
\begin{eqnarray}\label{f:energy-B^k-Delta}
\E_{k+1}(A,
B)&=&\sum_x(A\c A)(x)(B\c B)(x)^{k}\nonumber \\
&=&\sum_{x_1,\dots, x_{k-1}}\Big (\sum_y A(y)B(y+x_1)\dots
B(y+x_{k})\Big )^2 =\E(\Delta_k (A),B^{k}) \,,
 \end{eqnarray}
where
$$
    \Delta (A) = \Delta_k (A) := \{ (a,a, \dots, a)\in A^k \}\,.
$$
We also put $\Delta(x) = \Delta (\{ x \})$, $x\in \Gr$.


Quantities $\E_k (A,B)$ can be written in terms of generalized convolutions.

\begin{definition}
   Let $k\ge 2$ be a positive number, and $f_0,\dots,f_{k-1} : \Gr \to \C$ be functions.
Write $F$ for the vector $(f_0,\dots,f_{k-1})$ and $x$ for vector $(x_1,\dots,x_{k-1})$.
Denote by
$${\mathcal C}_k (f_0,\dots,f_{k-1}) (x_1,\dots, x_{k-1})$$
the function
$$
    \Cf_k(F) (x) =  \Cf_k (f_0,\dots,f_{k-1}) (x_1,\dots, x_{k-1}) = \sum_z f_0 (z) f_1 (z+x_1) \dots f_{k-1} (z+x_{k-1}) \,.
$$
Thus, $\Cf_2 (f_1,f_2) (x) = (f_1 \circ f_2) (x)$.
If $f_1=\dots=f_k=f$ then write
$\Cf_k (f) (x_1,\dots, x_{k-1})$ for $\Cf_k (f_1,\dots,f_{k}) (x_1,\dots, x_{k-1})$.
\end{definition}

In particular, $(\Delta_k (B) \c A^k) (x_1,\dots,x_k) = \Cf_{k+1} (B,A,\dots,A) (x_1,\dots,x_k)$, $k\ge 1$.

For a positive integer $n,$ we set $[n]=\{1,\ldots,n\}$.
All logarithms used in the paper are to base $2.$
By  $\ll$ and $\gg$ we denote the usual Vinogradov's symbols.
If $p$ is a prime number then write $\F_p$ for $\Z/p\Z$ and
$\F^*_p$ for $(\Z/p\Z) \setminus \{ 0 \}$.


\section{Preliminaries}
\label{sec:preliminaries}

Suppose that $l,k \ge 2$ be positive integers and ${\bf F} = (f_{ij})$, $i=0,\dots,l-1;j=0,\dots,k-1$
be a functional matrix, $f_{ij} : \Gr \to \C$.
Let $R_0,\dots,R_{l-1}$ and $C_0,\dots,C_{k-1}$ be rows and columns of the matrix, correspondingly.
The following commutative relation holds.

\begin{lemma}
    For any positive integers $l,k \ge 2$, we have
    \begin{equation}\label{f:commutative_C}
        \Cf_l (\Cf_k (R_0), \dots, \Cf_k (R_{l-1})) = \Cf_k (\Cf_l (C_0), \dots, \Cf_l (C_{k-1})) \,.
    \end{equation}
\label{l:commutative_C}
\end{lemma}
\begin{proof}
Let $y^{(i)} = (y_{i1}, \dots,y_{i(k-1)})$, $i\in [l-1]$,
and $y_{(j)} = (y_{1j}, \dots,y_{(l-1)j})$, $j\in [k-1]$.
Put also $y_{0j} = 0$, $j=0,\dots,k-1$, $y_{i0}=0$, $i=1,\dots,l-1$ and  $x_0=0$.
We have
$$
    \Cf_l (\Cf_k (R_0), \dots, \Cf_k (R_{l-1})) (y^{(1)},\dots,y^{(l-1)})
        =
$$
$$
        =
            \sum_{x_1,\dots,x_{k-1}} \Cf_k (R_0) (x_1,\dots,x_{k-1}) \Cf_k (R_1) (x_1 + y_{11},\dots,x_{k-1} + y_{1(k-1)})
                \dots
$$
$$
                \dots
                    \Cf_k (R_{l-1}) (x_1 + y_{(l-1)1},\dots,x_{k-1} + y_{(l-1)(k-1)})
        =
            \sum_{x_0,\dots,x_{k-1}} \sum_{z_0,\dots,z_{l-1}} \prod_{i=0}^{l-1} \prod_{j=0}^{k-1} f_{ij}(x_j+y_{ij}+z_i) \,.
$$
Changing the summation, we obtain
$$
    \Cf_l (\Cf_k (R_0), \dots, \Cf_k (R_{l-1})) (y^{(1)},\dots,y^{(l-1)})
        =
$$
$$
        =
            \sum_{z_1,\dots,z_{l-1}} \Cf_l(C_0) (z_1,\dots,z_{l-1})
                \Cf_l(C_1) (z_1+y_{11},\dots,z_{l-1}+y_{(l-1)1}) \dots
$$
$$
    \dots
                \Cf_l(C_{l-1}) (z_1+y_{1(k-1)},\dots,z_{l-1}+y_{(l-1)(k-1)})
                    =
                        \Cf_k (\Cf_l (C_0), \dots, \Cf_l (C_{k-1})) (y_{(1)}, \dots, y_{(k-1)}) \,.
$$
as required.
$\hfill\Box$
\end{proof}

\begin{corollary}
    For any functions the following holds
$$
    \sum_{x_1,\dots, x_{l-1}} \Cf_l (f_0,\dots,f_{l-1}) (x_1,\dots, x_{l-1})\, \Cf_l (g_0,\dots,g_{l-1}) (x_1,\dots, x_{l-1})
        =
$$
\begin{equation}\label{f:scalar_C}
        =
        \sum_z (f_0 \circ g_0) (z) \dots (f_{l-1} \circ g_{l-1}) (z) \quad \quad  \mbox{\bf (scalar product), }
\end{equation}
moreover
$$
    \sum_{x_1,\dots, x_{l-1}} \Cf_l (f_0) (x_1,\dots, x_{l-1}) \dots  \Cf_l (f_{k-1}) (x_1,\dots, x_{l-1})
        =
$$
\begin{equation}\label{f:gen_C}
        =
            \sum_{y_1,\dots,y_{k-1}} \Cf^l_k (f_0,\dots,f_{k-1}) (y_1,\dots,y_{k-1})
                \quad \quad  \mbox{\bf (multi--scalar product), }
\end{equation}
    and
$$
    \sum_{x_1,\dots, x_{l-1}} \Cf_l (f_0) (x_1,\dots, x_{l-1})\, (\Cf_l (f_1) \circ \dots \circ \Cf_l (f_{k-1})) (x_1,\dots, x_{l-1})
        =
$$
\begin{equation}\label{f:conv_C}
        =
            \sum_{z} (f_0 \circ \dots \circ f_{k-1})^l (z)
                \quad \quad  \mbox{\bf (} \sigma_{k} \quad \mbox{\bf for } \quad \Cf_l \mbox{\bf )}  \,.
\end{equation}
\label{c:commutative_C}
\end{corollary}
\begin{proof}
Take $k=2$ in (\ref{f:commutative_C}).
Thus ${\bf F}$ is a $l\times 2$ matrix in the case.
We have
$$
    \Cf_l (f_0 \circ g_0,\dots,f_{l-1} \circ g_{l-1}) (x_1,\dots, x_{l-1})
        =
            (\Cf_l (f_0,\dots,f_{l-1}) \circ \Cf_l (g_0,\dots,g_{l-1})) (x_1,\dots, x_{l-1}) \,.
$$
Putting $x_j=0$, $j\in [l-1]$, we obtain (\ref{f:scalar_C}).
Applying the last formula $(k-2)$ times and after that formula (\ref{f:scalar_C}), we get (\ref{f:conv_C}).
Finally, taking ${\bf F}_{ij} = f_j$, $i=0,\dots,l-1;j=0,\dots,k-1$
and putting all variables in (\ref{f:commutative_C}) equal zero, we obtain (\ref{f:gen_C}).
This completes the proof.
$\hfill\Box$
\end{proof}

\bigskip

We need in the Balog--Szemer\'{e}di--Gowers theorem in the symmetric form, see \cite{tv} section 2.5.

\begin{theorem}
    Let $A,B\subseteq \Gr$ be two sets, $K\ge 1$ and $\E(A,B) \ge |A|^{3/2} |B|^{3/2} / K$.
    Then there are $A'\subseteq A$, $B' \subseteq B$ such that
    $$
        |A'| \gg |A| / K\,, \quad \quad |B'| \gg |B| / K \,,
    $$
    and
    $$
        |A' + B'| \ll K^7 |A|^{1/2} |B|^{1/2} \,.
    $$
\label{t:BSzG_1.5}
\end{theorem}

Now let $\Gr = \F_p$, where $p$ is a prime number.
In the situation the following lemma which is a consequence of Stepanov's approach \cite{Stepanov}
can be formulated (see, e.g. \cite{sv}).

\begin{lemma}
{
    Let $p$ be a prime number,
    $\Gamma\subseteq \F^*_p$ be a multiplicative subgroup, and
    $Q,Q_1,Q_2\subseteq \F^*_p$ be any $\Gamma$--invariant sets such that
    $|Q| |Q_1| |Q_2| \ll |\Gamma|^{5}$ and $|Q| |Q_1| |Q_2| |\Gamma| \ll p^3$.
    Then
        \begin{equation}\label{f:improved_Konyagin_old3}
            \sum_{x\in Q} (Q_1 \circ Q_2) (x) \ll |\Gamma|^{-1/3} (|Q||Q_1||Q_2|)^{2/3} \,.
        \end{equation}
}
\label{l:improved_Konyagin_old}
\end{lemma}

Using Lemma \ref{l:improved_Konyagin_old}, one can easily deduce
upper bounds for moments of convolution of $\Gamma$ (see, e.g. \cite{ss}).

\begin{corollary}
    Let $p$ be a prime number and $\G \subseteq \F_p^*$ be a multiplicative subgroup, $|\G| \ll p^{2/3}$.
    Then
    \begin{equation}\label{f:E_2_E_3}
        \E(\Gamma) \ll |\Gamma|^{5/2} \,, \quad \E_3 (\G) \ll |\G|^3 \log |\G| \,,
    \end{equation}
    and for all $l\ge 4$ the following holds
    \begin{equation}\label{f:E_l}
        \E_l (\G) = |\G|^l + O(|\G|^{\frac{2l+3}{3}}) \,.
    \end{equation}
\label{cor:E_l}
\end{corollary}

Certainly, the condition $|\G| \ll p^{2/3}$ in formula (\ref{f:E_l}) can be  relaxed.

The same method gives a generalization (see \cite{K_Tula}).

\begin{theorem}
    Let $\G \subseteq \F^*_p$ be a multiplicative subgroup, $|\G| < \sqrt{p}$.
    Let also $d\ge 2$ be a positive integer.
    Then arranging $(\G *_{d-1} \G) (\xi_1) \ge (\G *_{d-1} \G) (\xi_2) \ge \dots $, where $\xi_j \neq 0$
    belong to distinct cosets,
    we have
    $$
        (\G *_{d-1} \G) (\xi_j) \ll_d |\G|^{d-2 + 3^{-1} (1+2^{2-d}) } j^{-\frac{1}{3}} \,.
    $$
    In particular
    \begin{equation}\label{}
        \T_d (\G) \ll_d |\G|^{2d-2+2^{1-d}} \,,
    \end{equation}
    further
    \begin{equation}\label{}
        \sum_z (\G\circ_{d-1} \G)^3 (z) \ll_d |\G|^{3d-4+2^{2-d}} \cdot \log |\G| \,,
    \end{equation}
    and similar
    \begin{equation}\label{f:E_3_gen_1}
        \sum_z (\G \c \G) (z) ((\G *_{d-1} \G) \c (\G *_{d-1} \G))^2 (z)
            \ll_d |\G|^{4d-2+3^{-1} (1+2^{3-2d}) } \cdot \log |\G| \,.
    \end{equation}
\label{t:3_d_moment}
\end{theorem}

We need in a
lemma about Fourier coefficients of an arbitrary $\G$--invariant set (see e.g. \cite{ss}).

\begin{lemma}
        Let $\G \subseteq \F^*_p$ be a multiplicative subgroup,
        and $Q$ be an $\G$--invariant subset of $\F^*_p$,
        that is $Q\G=Q$.
        Then for any $\xi \neq 0$ the following holds
\begin{equation}\label{f:G-inv_bound_F}
    | \FF{Q} (\xi) | \le \min \left\{ \left(\frac{|Q|p}{|\G|}\right)^{1/2} \,, \frac{|Q|^{3/4} p^{1/4} \E^{1/4} (\G)}{|\G|} \,,
                            p^{1/8} \E^{1/8} (\G) \E^{1/8} (Q) \left(\frac{|Q|}{|\G|}\right)^{1/2} \right\} \,.
\end{equation}
\label{l:G-inv_bound_F}
\end{lemma}

\bigskip

Recall that a set $A=\{a_1,\dots, a_n\} \subseteq \R$ is called {\it convex}
if $a_i-a_{i-1}<a_{i+1}-a_i$ for every $2\le i\le n-1.$
Convex sets have statistics similar to multiplicative subgroups, in some sense.
We need in a lemma, see e.g. \cite{ss2} or \cite{Li}.

\begin{lemma}
    Let $A$ be a convex set, and $B$ be an arbitrary set.
    Then
    $$
        \E_3 (A) \ll |A|^3 \log |A| \,,
    $$
    and
    $$
        \E(A,B) \ll |A| |B|^{\frac{3}{2}} \,.
    $$
\label{l:E_3_convex}
\end{lemma}

Now consider quantities $(A*_{k-1} A) (x)$.
By a classical result of Andrews \cite{Andrews}, we have for any $x$ that
$$
    (A*_{k-1} A) (x) \ll_k |A|^{\frac{k (k-1)}{k+1}} \,.
$$
The following result was proved in \cite{ikrt}.

\begin{theorem}
    Let $A$ be a convex set, and $k\ge 2$ be an integer.
    Then arranging $(A*_{k-1} A) (x_1) \ge (A*_{k-1} A) (x_2) \ge \dots $, we have
    \begin{equation}\label{f:E_3_gen_2-}
        (A*_{k-1} A) (x_j) \ll_k |A|^{k-\frac{4}{3} (1-2^{-k})} j^{-\frac{1}{3}} \,.
    \end{equation}
    In particular
    \begin{equation}\label{f:E_3_gen_2}
        \sum_x (A \c A) (z) ((A *_{k-1} A) \c (A *_{k-1} A))^2 (x)
            \ll_k |A|^{4k-2+3^{-1} (1+2^{3-2k}) } \cdot \log |A| \,.
    \end{equation}
\label{t:convex_ikrt}
\end{theorem}

As was realized by Li \cite{Li} (see also \cite{ss_E_k}) that subsets $A$ of real numbers
with small multiplicative doubling looks like convex sets.
More precisely, the following lemma from \cite{ss_E_k} holds.

\begin{lemma}
    Let $A,B \subseteq \R$ be finite sets and let $|AA| = M |A|$.
    Then arranging $(A\c B) (x_1) \ge (A\c B) (x_2) \ge \dots $, we have
$$
    (A\c B) (x_j) \ll (M \log M)^{2/3} |A|^{1/3} |B|^{2/3} j^{-1/3} \,.
$$
In particular
$$
    \E (A,B) \ll M \log M |A| |B|^{3/2} \,.
$$
\label{l:arranging_gen}
\end{lemma}

\section{Weighted Katz--Koester transform}
\label{sec:weighted_KK}

In the section we have deal with so--called Katz--Koester trick   \cite{kk} based on inequality (\ref{f:kk_introduction}),
which has recently found many applications, see \cite{A(A+1),Li,Li2,Rudnev_C,Schoen_Freiman,ss,ss2,ss_E_k,sv}.
We collect all required tools in the section.

First of all let us recall Lemma 2.4 and Corollary 2.5 from \cite{sv}.
We gather the results in the following proposition.

\begin{proposition}
    Let $k \ge 2$, $m \in [k]$ be positive integers, and
    let $A_1,\dots,A_k,B$ be finite subsets of an abelian group.
    Then
    \begin{equation}\label{f:characteristic1}
        A_1 \m \dots \m A_k - \Delta_k (B) = \{ (x_1,\dots,x_k) ~:~ B \cap (A_1-x_1) \cap \dots \cap(A_k - x_k) \neq \emptyset \}
    \end{equation}
    and
    \begin{equation}\label{f:characteristic2}
        A_1 \m \dots \m A_k - \Delta_k (B)
            =
    \end{equation}
    $$
                \bigcup_{(x_1,\dots,x_m) \in A_1 \m \dots \m A_m - \Delta(B)}
                    \{ (x_1,\dots,x_m) \} \m (A_{m+1} \m \dots \m A_k - \Delta_{k-m} (B \cap (A_1-x_1) \cap \dots \cap (A_m-x_m)) \,.
    $$
\label{p:characteristic}
\end{proposition}

Let $A,B\subseteq \Gr$ be sets, $x\in \Gr^k$, $s\in \Gr^l$.
By the proposition, we have $x \in A^k - \Delta_k (A^B_s)$ iff $s \in A^l - \Delta_l (A^B_x)$
because of $x \in A^k - \Delta_k (A^B_s )$ iff $A^B_x \bigcap A^B_s \neq \emptyset$.
Hence, we obtain the following formula
\begin{equation}\label{f:x_in_A-A_s}
    \sum_{s\in A^l - \Delta_l (B)} (A^k-\Delta_k (A^B_s)) (x) = |A^l- \Delta_l (A^B_x)| \,.
\end{equation}
In particular
$$
    (A-A_s) (x) = (A-A_x) (s) \quad \mbox{ and } \quad \sum_s (A-A_s) (x) = |A-A_x| \,.
$$

The next lemma is a very special case of Lemma 2.8 from \cite{sv}.

\begin{lemma}
    Let $A,B\subseteq \Gr$ be sets, and $k$, $l$ be positive integers.
    Then
    $$
        \sum_{s\in \Gr^l} \E (A^k, \Delta(A^B_s)) = \E_{k+l+1} (B,A) \,.
    $$
\label{l:E(A,A_s)_special}
\end{lemma}

Now we obtain the main proposition of the section.

\begin{proposition}
    Let
    $A,B\subseteq \Gr$ be two sets,
    $k$, $l$ be positive integers,
    and $q : \Gr^k \to \mathbb {C}$ be an arbitrary function.
    Then
\begin{equation}\label{f:wight_1}
    |A|^{2l} \left| \sum_{x \in \Gr^k} q(x) (A^k \circ \Delta_k (B)) (x) \right|^2
        \le
            \E_{k+l+1} (B,A) \cdot \sum_{x \in \Gr^k} |A^l \pm \Delta_l (A^B_x)| |q(x)|^2 \,.
\end{equation}
\label{p:weight}
\end{proposition}
\begin{proof}
We have
\begin{equation}\label{f:first_p:weight}
    \sum_s \sum_{x} (A^k \circ \Delta (A^B_{s})) (x) q(x)
        = \sum_x q(x) \sum_s (A^k \circ \Delta (A^B_{s})) (x)
            = |A|^l \sum_x q(x) (A^k \circ \Delta (B)) (x) \,.
\end{equation}
Applying  Cauchy--Schwartz twice, Lemma \ref{l:E(A,A_s)_special} and formula (\ref{f:x_in_A-A_s}), we get
$$
    |A|^{2l} \left| \sum_x q(x) (A^k \circ \Delta (B)) (x) \right|^2
            \le
$$
$$
            \le
        \left( \sum_s \left( \sum_{x} (A^k - \Delta (A^B_{s})) (x) |q(x)|^2 \right)^{1/2}
            \cdot \left( \sum_x (A^k \circ \Delta (A^B_s))^2 (x) \right)^{1/2} \right)^2
    \le
$$
$$
    \le
        \sum_{x} |q(x)|^2 \sum_s (A^k - \Delta (A^B_{s})) (x) \cdot \sum_s \E (A^k, \Delta(A^B_s))
            =
                \sum_{x} |q(x)|^2 |A^l - \Delta(A^B_x)| \cdot \E_{k+l+1} (B,A)
$$
and formula (\ref{f:wight_1}) with minus
follows.
To
get
the remain formula with plus consider $A^*_s = A^*_s (B) := B \cap (s_1-A) \cap \dots \cap (s_l-A)$ instead of $A^B_s$.
It is easy to see that formula (\ref{f:first_p:weight}) takes place for such sets.
Besides as in
Proposition \ref{p:characteristic}, we have
$x\in A^k-\Delta(A^*_s)$ iff $A^*_s \cap A-x_1 \cap \dots \cap A-x_k \neq \emptyset$
and further
iff $s\in A^l + \Delta(A^B_x)$.
Thus, we
obtain
an analog of formula (\ref{f:x_in_A-A_s})
$$
    \sum_s (A^k - \Delta (A^*_{s})) (x) = |A^l+\Delta(A^B_x)| \,.
$$
Finally,
$$
    \sum_s \E (A^k, \Delta(A^*_s)) = \sum_z (A\circ A)^{k} (x) (B\c B) (x) (A\circ A)^{l} (-x) = \E_{k+l+1} (B,A)
$$
and the result is proved.
$\hfill\Box$
\end{proof}

\bigskip

Let us derive simple consequences  of the result above.
Consider the case $A=B$.
If we take $k=l=1$ and $q(x) = (A-A)(x)$ then we obtain Corollary 3.2 from \cite{ss} as well as Lemma 2.3 from \cite{sv}.
If we take $k=l=1$ and $q(x) = (A\circ A)^{1/2} (x)$ then we get Lemma 2.5 from \cite{Li}.
Let us derive further consequences.

\begin{corollary}
Let
    $A,B\subseteq \Gr$ be two sets, and $k$, $l$ be positive integers.
    Then
\begin{equation}\label{f:energy_wight_1}
    |A|^{2l} \E^2_{k+1} (B,A)
        \le
            \E_{k+l+1} (B,A) \cdot \sum_{x} |A^l \pm \Delta(A^B_x)| (A^k \circ \Delta_k (B))^2 (x)
\end{equation}
and
\begin{equation}\label{f:energy_wight_2}
    |A|^{2l} \sum_x \frac{(A^k \circ \Delta_k (B))^2 (x)}{|A^l \pm \Delta_l (A^B_x)|}
        \le
            \E_{k+l+1} (B,A) \,.
\end{equation}
\label{cor:energy_weight}
\end{corollary}
\begin{proof}
Taking $q(x) = (A^k \circ \Delta (B)) (x)$ and applying Corollary \ref{c:commutative_C}, we obtain the first formula.
Choosing $q(x)$ optimally, that is $$q(x) = \frac{(A^k \circ \Delta_k (B)) (x)}{|A^l \pm \Delta_l (A^B_x)|} \,,$$
 we get (\ref{f:energy_wight_2}).
$\hfill\Box$
\end{proof}

\bigskip

Until the end of the section suppose, for simplicity, that $B=A$.
Corollary \ref{c:commutative_C} implies that
$\sum_x (A^k \circ \Delta (A))^2 (x) = \E_{k+1} (A)$.
Combining the identity with formula (\ref{f:energy_wight_2}),
we obtain

\begin{corollary}
\begin{equation}\label{tmp:14.10.2011_2}
    \sum_{x ~:~ |A^l \pm \Delta_l (A_x)| \ge \frac{|A|^{2l} \E_{k+1} (A)}{2 \E_{k+l+1} (A)}} (A^k \circ \Delta_k (A))^2 (x)
        \ge
            2^{-1} \E_{k+1} (A) \,.
\end{equation}
\end{corollary}

For example ($k=l=1$)
$$
    \sum_{x ~:~ |A\pm A_x| \ge 2^{-1} |A|^{2} \E (A) \E^{-1}_{3} (A)} |A_x|^2 \ge 2^{-1} \E(A) \,.
$$

Suppose that $\E_{k+l+1} (A) \ll |A|^{k+l+1}$.
Using a trivial bound $|A^l \pm \Delta(A_x)| \le |A|^{l} |A_x|$,
we see that the lower bound for $|A_x|$, deriving from (\ref{tmp:14.10.2011_2}),
namely, $|A_x| \ge 2^{-1} |A|^l \E_{k+1} (A) \E^{-1}_{k+l+1} (A)$
is potentially sharper then usual estimate $|A_x| \ge 2^{-1} \E_{k+1} (A) |A|^{-(k+1)}$,
which follows from the identity $\sum_{x} |A_x|^2 = \E_{k+1} (A)$.

The same arguments give

\begin{corollary}
\begin{equation}\label{f:cor_2_tmp}
    \sum_{x ~:~ |A^l \pm \Delta_l (A_x)| \ge (A^k \circ \Delta_k (A)) (x) \cdot \frac{|A|^{2l+k+1}}{2 \E_{k+l+1} (A)}}\,
            (A^k \circ \Delta_k (A)) (x)
        \ge
            2^{-1} |A|^{k+1} \,.
\end{equation}
\end{corollary}

In the case $k=l=1$, we obtain
$$
    \sum_{x ~:~ |A\pm A_x| \ge |A_x| \cdot \frac{|A|^4}{2\E_3(A)}}\, |A_x| \ge 2^{-1} |A|^2 \,.
$$

Finally in the case $k=l=1$, let us obtain an useful corollary.

\begin{corollary}
    Let $\a$, $p$ be real numbers, $p>1$.
    Then
\begin{equation}\label{f:a_p}
    \sum_x |A_x|^\a \le \left( \frac{\E_3 (A)}{|A|^2} \right)^{1/p}
        \cdot
            \left( \sum_x |A\pm A_x|^{\frac{1}{p-1}} |A_x|^{\frac{\a p-2}{p-1}} \right)^{(p-1)/p} \,.
\end{equation}
\label{cor:a_p}
\end{corollary}


\section{Eigenvalues of some operators}
\label{sec:eigenvalues}


We make use of some operators, which were introduced in \cite{s}.
These operators have found some applications in additive combinatorics and number theory
(see \cite{s} and \cite{ss_E_k}).

\begin{definition}
Let $\Gr$ be an abelian group, and $\_phi,\psi$ be two complex functions.
By $\oT^{\_phi}_\psi$ denote the following operator on the space of functions $\Gr^{\mathbb{C}}$
\begin{equation}\label{F:T}
    (\oT^{\_phi}_\psi f ) (x) = \psi(x) (\FF{\_phi^c} * f) (x) \,,
\end{equation}
where $f$ is an arbitrary complex function on $\Gr$.
\end{definition}

Suppose that $\Gr$ is a finite abelian group, and $A\subseteq \Gr$ is a set.
Denote by $\ov{\oT}^{\_phi}_A$ the restriction of operator $\oT^{\_phi}_A$ onto the space of the functions
with supports on $A$.
Recall some simple properties of operators $\ov{\oT}^{\_phi}_A$ which were obtained in \cite{s}.
First of all, it was proved, in particular,
that
operators $\oT^{\_phi}_A$ and $\ov{\oT}^{\_phi}_A$ have the same non--zero eigenvalues.
Second of all, if $\_phi$ is a real function then the operator $\ov{\oT}^{\_phi}_{A}$ is
symmetric (hermitian)
and
if $\_phi$ is a nonnegative function then the operator is nonnegative definite.
The action of $\ov{\oT}^{\_phi}_{A}$ can be written as
\begin{equation}\label{F:T_S_action}
    \langle \ov{\oT}^{\_phi}_A u, v \rangle
        =
            \sum_x (\FF{\_phi^c} * u) (x) \ov{v} (x)
                =
                    \sum_x \FF{\_phi^c} (x) (u \c \ov{v}) (x)
                =
                    \sum_x \_phi (x) \FF{u} (x) \ov{\FF{v} (x)} \,,
\end{equation}
where $u,v$
are
arbitrary
functions such that $\supp u, \supp v \subseteq A$.
Further
\begin{equation}\label{F:trace}
    \tr ( \ov{\oT}^{\_phi}_{A} )
        =
            |A| \FF{\_phi} (0)
                =
                    \sum_{j=1}^{|A|} \mu_j ( \ov{\oT}^{\_phi}_{A} )
                        =
                            \sum_{j=1}^{|\Gr|} \mu_j ( \oT^{\_phi}_{A} ) \,.
\end{equation}
If $\_phi$ is a real function then as was noted before
$\ov{\oT}^{\_phi}_{A}$ is a symmetric matrix.
In particular, it is a normal matrix and we get
\begin{equation}\label{F:trace_sq}
    \tr ( \ov{\oT}^{\_phi}_{A} ( \ov{\oT}^{\_phi}_{A} )^*)
        =
            \sum_z |\FF{\_phi} (z)|^2 (A \circ A) (z)
                =
                    \sum_z (\_phi \circ \_phi) (z) |\FF{A} (z)|^2
                =
                    \sum_{j=1}^{|A|} \mu^2_j ( \ov{\oT}^{\_phi}_{A} )
                        =
                            \sum_{j=1}^{|\Gr|} \mu^2_j ( \oT^{\_phi}_{A} ) \,.
\end{equation}
We will deal with just nonnegative definite symmetric operators.
In the case we arrange the eigenvalues in order of magnitude
$$
    \mu_0 (\ov{\oT}^{\_phi}_{A}) \ge \mu_1 (\ov{\oT}^{\_phi}_{A}) \ge \dots \ge \mu_{|A|-1} (\ov{\oT}^{\_phi}_{A}) \,.
$$
Further properties of such  operators  can be found in \cite{s}.
The connection of such operators with higher energies $\E_k (A)$ is discussed in \cite{ss_E_k}.

\bigskip

Now we consider the situation when $A$ equals some multiplicative subgroup.
It turns out that in this case we know all eigenvalues $\mu_j$ as well as all
eigenfunctions.

Let $p$ be a prime number, $q=p^s$ for some integer $s \ge 1$.
Let
$\F_q$ be the field with  $q$ elements, and let $\Gamma\subseteq \F_q$ be a multiplicative subgroup.
We will write $\F^*_q$ for $\F_q\setminus \{ 0 \}$.
Denote by $t$ the cardinality of $\Gamma$, and put $n=(q-1)/t$. Let
also $\mathbf{g}$ be a primitive root, then $\Gamma = \{ \mathbf{g}^{nl} \}_{l=0,1,\dots,t-1}$.
Let $\chi_\a (x)$, $\a \in [t]$ be the
orthonormal
family of multiplicative characters on $\Gamma$, that is
\begin{equation}\label{f:chi_Gamma}
    \chi_\a (x) = |\G|^{-1/2} \cdot \Gamma(x) e\left( \frac{\a l}{t} \right) \,, \quad x=\mathbf{g}^{nl} \,, \quad 0\le l < t \,.
\end{equation}
Clearly, products of such functions form a basis on Cartesian products of $\G$.

The following proposition was obtained, basically,
in \cite{ss_E_k} (except formula (\ref{f:C_3_lower_bound})).
We recall the proof for the sake of completeness.

\begin{proposition}
    Let $\Gamma\subseteq \F^*_q$ be a multiplicative subgroup.
    If $\psi$ is an arbitrary $\G$--invariant function
    then the functions $\chi_\a (x)$ are eigenfunctions of the operator $\ov{\oT}^{\FF{\psi}}_{\Gamma}$.
    Suppose, in addition, that
    $\FF{\psi} (x) \ge 0$.
   Then for any functions $u : \F_q \to \C$ and $v : \F_q \to \R^+$ the following holds
   \begin{equation}\label{f:C_3_lower_bound}
        \sum_{x,y\in \G} \psi(x-y) \Cf_3 (v,\ov{u}, u) (x,y)
            \ge
                |\Gamma|^{-2} \sum_{x} \psi (x) (\Gamma \circ \Gamma) (x)
                    \cdot
                        \sum_{x,y \in \G} \Cf_3 (v,\ov{u}, u) (x,y) \,.
   \end{equation}
    In particular,
    for any function $u$ with support on $\Gamma$, we have
    \begin{equation}\label{f:connected}
        \sum_{x} \psi (x) (u \circ \ov{u}) (x)
            \ge
                |\Gamma|^{-2}
                    \sum_{x} \psi (x) (\Gamma \circ \Gamma) (x)
                        \cdot
                    \Big| \sum_{x\in \Gamma} u(x) \Big|^2 \,.
   \end{equation}
    \label{p:eigenfunctions_Gamma}
\end{proposition}
\begin{proof}
We have to show that
$$
    \mu f(x) = \Gamma(x) (\psi * f) (x) \,, \quad \mu \in \mathbb{C}
$$
for $f(x) = \chi_\a (x)$.
For every $\gamma\in \Gamma$, we obtain
\begin{eqnarray}\label{f:Schoen}
    (\psi * f) (\gamma x) &=& \sum_z f(z) \psi (\gamma x -z) = \sum_z f(\gamma z) \psi
    (\gamma x -\gamma z)\\
    &=&
    f(\gamma) \cdot \sum_z f(z) \psi (x-z) = f(\gamma) \cdot (\psi * f) (x)
\end{eqnarray}
as required.

Formula (\ref{f:connected}) follows from (\ref{f:C_3_lower_bound}) if one take $v=\d_0$.
We give another independent proof.
Because of $\FF{\psi} (x) \ge 0$ the operator $\ov{\oT}^{\FF{\psi}}_{\Gamma}$
is symmetric and nonnegative definite.
Thus all its eigenvalues are nonnegative.
Put $\_phi = q^{-1} \FF{\psi}$.
If $u=\sum_{\a} c_\a \chi_\a$ then
$$
    \langle \ov{\oT}^{\_phi}_{\Gamma} u, u \rangle
        =
            \sum_{x} \psi (x) (u \circ \ov{u}) (x)
                =
                    \sum_\a |c_\a|^2 \mu_\a (\ov{\oT}^{\_phi}_{\Gamma})
                        \ge
                            |\Gamma|^{-2} \langle u, \Gamma \rangle^2
                                \sum_{x} \psi (x) (\Gamma \circ \Gamma) (x)
$$
and we obtain (\ref{f:connected}).

Finally, for any function $F : \G\m \G \to \C$, we have
$$
    F(x,y) = \sum_{\a,\beta} c_{\a,\beta} (F) \chi_\a (x) \chi_\beta (y) \,.
$$
Thus
$$
    \sum_{x,y} F(x,y) \psi(x-y) = \sum_\a \mu_\a \cdot c_{-\a,\a} (F)
$$
and we just need to check that $c_{-\a,\a} (F) \ge 0$ for $F(x,y)=\Cf_3 (v,\ov{u},u) (x,y)$.
By assumption $v\ge 0$.
Hence by Corollary \ref{c:commutative_C}
\begin{equation}\label{tmp:21.07.2012_1'}
    c_{-\a,\a} (F) = \sum_{x,y} F(x,y) \ov{\chi_\a (x)} \chi_\a (y) = \sum_z v(z) |(\chi_\a \c u)|^2 (z) \ge 0
\end{equation}
and the result follows.
$\hfill\Box$
\end{proof}

\bigskip

In particular, for any $k\ge 1$
\begin{equation}\label{f:E_for_subgroups}
    \E_{k+1} (\G) = \max_{ f~:~ \supp f\subseteq \G,\, \| f \|_2^2 = |\G|}\,\,\, \sum_x (\G \c \G)^k (x) (f \c f) (x) \,.
\end{equation}

\begin{remark}
    It is not difficult to replace a multiplicative subgroup $\G$
    in the previous proposition onto arbitrary
    coset
    (see \cite{ss_E_k}).
    Indeed,
    for every $\xi \in \F^*_q / \Gamma$ and $\a\in [|\Gamma|],$ let us
define the functions $\chi^{\xi}_\a (x) := \chi_\a (\xi^{-1} x).$
Then, clearly,   $\supp \chi^{\xi}_\a=\xi \cdot \Gamma$ and
$\chi^{\xi}_\a (\gamma x) = \chi_\a (\gamma) \chi^{\xi}_\a (x)$ for all
$\gamma \in \Gamma$. Using the argument from  Proposition
\ref{p:eigenfunctions_Gamma} it is easy to see that the functions
$\chi^{\xi}_\a$ are orthonormal eigenfunctions of the operator
$\ov{\oT}^{\FF{\psi}}_{\xi \Gamma}$.
Thus, we can replace $\G$ onto $\xi \Gamma$.
\label{note:cosets_eigen}
\end{remark}

Proposition \ref{p:eigenfunctions_Gamma} has an interesting corollary about
Fourier coefficients of functions with supports on $\G$.
In particular, it gives exact formula for exponential sums over multiplicative subgroups.

\begin{corollary}
    Let $\Gamma\subseteq \F^*_q$ be a multiplicative subgroup.
    Suppose that $u$ is a function with support on $\G$.
    Then for any $\lambda \in \F_q$ the following holds
\begin{equation}\label{f:exact_Fourier_formula}
    |\FF{u} (\lambda)|^2 = |\G|^2 \cdot \min_{h}
        \frac{\sum_x |\FF{h} (x)|^2 |\FF{u}(x+\lambda)|^2}{\sum_x |\FF{h} (x)|^2 |\FF{\G}(x)|^2} \,,
\end{equation}
    and, in addition,  for any $v:\F_q \to \R^{+}$, we have
\begin{equation}\label{f:exact_Fourier_formula_C_3}
    \sum_{x,y \in \G} \Cf_3 (v,u,\ov{u}) (x,y)
        = |\G|^2 \cdot \min_{h}
        \E^{-1}(h,\G)  \cdot \sum_{x,y\in \G} (h\circ \ov{h}) (x-y) \Cf_3 (v,u,\ov{u}) (x,y) \,,
\end{equation}
    where the minimum is taken over all nonzero $\G$--invariant functions.
\end{corollary}
\begin{proof}
Taking $\psi = h\circ \ov{h}$ in formula (\ref{f:connected}) of Proposition \ref{p:eigenfunctions_Gamma}
and using Fourier transform,
we obtain that
\begin{equation}\label{tmp:06.11.2011_1}
    |\sum_{z\in \G} u(z)|^2 \le |\G|^2 \cdot \min_{h}
        \frac{\sum_x |\FF{h} (x)|^2 |\FF{u}(x)|^2}{\sum_x |\FF{h} (x)|^2 |\FF{\G}(x)|^2}
\end{equation}
for any function $u$ with support on $\G$.
Considering $h \equiv 1$ we make sure that formula (\ref{tmp:06.11.2011_1}) is actually equality.
Now taking $u(x) e(-\lambda x)$ instead of $u(x)$, we have formula (\ref{f:exact_Fourier_formula}).
Equality (\ref{f:exact_Fourier_formula_C_3}) is a consequence of (\ref{f:C_3_lower_bound})
and can be obtained by similar arguments.
This completes the proof.
$\hfill\Box$
\end{proof}

Let $g : \F_q \to \C$ be a $\G$---invariant function.
It is convenient to write $\mu_\a (g)$ for $\mu_\a (\oT^{q^{-1} \FF{g}}_\G)$.
It is easy to see that
$\ov{\mu_\a (g)} = \mu_\a (\ov{g}^c) = \mu_{-\a} (\ov{g})$.
Multiplicative properties of the functions $\chi_\a$
allow us to prove formula (\ref{f:conv_e_1}) below,
which shows that the numbers $\mu_\a (\ov{g}h)$ and $\mu_\a (g)$, $\mu_\a (h)$ are
connected.

\begin{proposition}
    Let $g,h : \F_q \to \C$ be two $\G$---invariant functions.
    Then
\begin{equation}\label{f:conv_e_1}
    \mu_\a (\ov{g} h) = \frac{1}{|\G|} \sum_\beta \ov{\mu}_\beta (g) \mu_{\a+\beta} (h) = (\mu(\ov{g}) * \mu (h)) (\a) \,,
\end{equation}
    and
\begin{equation}\label{f:conv_e_2}
    \mu_\a (g) = |\G|^{1/2} \sum_x g(x) \chi_\a (1-x) \,.
\end{equation}
\label{p:convolution_of_eigenvalues}
\end{proposition}
\begin{proof}
We have
$$
    \frac{1}{|\G|} \sum_\beta \ov{\mu}_\beta (g) \mu_{\a+\beta} (h)
        =
            \frac{1}{|\G|} \sum_{x,y} \ov{g}(x) h(y)
                \sum_\beta (\ov{\chi}_\beta \c \chi_\beta )(x)
                    (\chi_{\a+\beta} \c \ov{\chi}_{\a+\beta} )(y)
                        =
$$
$$
                        =
     \frac{1}{|\G|} \sum_{x,y} \ov{g}(x) h(y)
        \sum_{z,w\in \Gamma}
            \sum_\beta \ov{\chi}_\beta (z) \chi_\beta (z+x) \chi_{\a+\beta} (w) \ov{\chi}_{\a+\beta} (w+y)
                =
$$
$$
    =
        \frac{1}{|\G|} \sum_{x,y} \ov{g}(x) h(y)
            \sum_{w \in \G} \chi_{\a} (w) \ov{\chi}_{\a} (w+y) \varpi (x,y,w) \,,
$$
where
$
    \varpi (x,y,w)
$
equals $1$ iff $w,w+y \in \G$ and, more importantly, $(z+x)/z = (w+y)/w$ for some $z$ such that $z,z+x \in \G$.
It is easy to see that the last situation appears exactly when $xy^{-1} \in \G$, provided by $y\neq 0$.
Besides $y=0$ iff $x=0$.
Thus by $\G$--invariance of the function $g$
$$
    \frac{1}{|\G|} \sum_\beta \ov{\mu}_\beta (g) \mu_{\a+\beta} (h)
        =
            \ov{g}(0) h(0)
                +
                    \frac{1}{|\G|} \sum_{x \neq 0,\, y \neq 0} \ov{g}(x) h(y)
                        \G(xy^{-1}) (\chi_\a \c \ov{\chi}_\a)(y)
    =
$$
$$
    =
        \ov{g}(0) h(0) + \sum_{y \neq 0} \ov{g}(y) h(y) (\chi_\a \c \ov{\chi}_\a)(y)
            =
                \sum_{y} \ov{g}(y) h(y) (\chi_\a \c \ov{\chi}_\a)(y)
                    =
                        \mu_\a (\ov{g} h)
$$
and we obtain formula (\ref{f:conv_e_1}).

One can derive (\ref{f:conv_e_2}) from (\ref{f:conv_e_1}).
Another way is to use formula (\ref{f:Schoen}) of Proposition \ref{p:eigenfunctions_Gamma}.
We propose
one more
variant.
Consider $\mu_\a (g) = f(\a)$ as a function on $\a$
and compute the Fourier transform of $f$.
Now write $e(x)$ for $e^{2\pi i x/|\G|}$.
We have for $\a \neq 0$
$$
    \FF{f} (\a) = \sum_\beta \sum_x g(x) \sum_z \chi_\beta (z-x) \ov{\chi}_\beta (z) e(-\a \beta)
        =
            \sum_x g(x) \G(x (1-\mathbf{g}^{n\a})^{-1}) =
$$
$$
    =
        \sum_x g(x (1-\mathbf{g}^{n\a})) \G(x) = |\G| g(1-\mathbf{g}^{n\a}) \,.
$$
Besides
the last formula holds
in the case
$\a=0$ because we have general identity (\ref{F:trace}).
Finally, using the inverse formula (\ref{f:inverse}), we obtain
$$
    \mu_\a (g) = \sum_\beta g(1-\mathbf{g}^{n\beta}) e(\a \beta)
        = |\G|^{1/2} \sum_x g(1-x) \chi_\a (x) = |\G|^{1/2} \sum_x g(x) \chi_\a (1-x) \,.
$$
This completes the proof.
$\hfill\Box$
\end{proof}

In particular, taking $\a=0$, $l=2$ and $g=h$ in formula (\ref{f:conv_e_1}),
we obtain formula (\ref{F:trace_sq}) for operators $\ov{\oT}^{\_phi}_\G$,
where $\_phi (x) = q^{-1}\FF{g}$ and $\G$ is a multiplicative subgroup.

\bigskip

\begin{corollary}
    Let $g : \F_q \to \R$ be a $\G$--invariant function.
    Put $\mu (\a) = \mu_\a (g)$.
    Then for all positive integers $l$, we have
\begin{equation}\label{f:mu_l_&_mu}
    \mu_\a (g^l) = (\mu *_{l-1} \mu  ) (\a) \,,
\end{equation}
and
\begin{equation}\label{}
    g^l (x-y) = \sum_\a (\mu *_{l-1} \mu  ) (\a) \chi_\a (x) \ov{\chi_\a (y)} \,,\quad x,y\in \G \,,
\end{equation}
where $*$ the normalized convolution
over $|\G|$.
In particular, numbers $\E(\G,\chi_\a)$, $\a\in [|\G|]$ determine $\E_l (\G)$ for all $l \ge 2$.
\end{corollary}

Now consider for a moment the case of prime $q=p$.

\begin{remark}
Suppose that  $g(x)=(\G \c \G)(x)$ and
$\mu_l (\a) = \mu_\a (g^l)$.
By Corollary \ref{cor:E_l}
and formulas (\ref{F:trace}), (\ref{F:trace_sq}), we get
for any $|\G| \ll p^{2/3}$ and $l\ge 2$ that
$$
    \sum_\a (\mu_l (\a) - |\G|^l)^2 \ll |\G|^{1+(2l+1) \cdot 2/3} = |\G|^{4l/3 + 5/3} \,.
$$
Thus, we have an asymptotic formula  for all $l\ge 2$
$$
    \mu_l (\a) = \sum_x (\G \c \G)^l (x) (\chi_\a \c \chi_\a) (x) = |\G|^l + O(|\G|^{2l/3 + 5/6}) \,, \quad \a \in [|\G|] \,.
$$
\end{remark}

\bigskip

Using the arguments from the proof of Proposition \ref{p:eigenfunctions_Gamma}, we obtain a general inequality.

\begin{proposition}
    Let $A\subseteq \Gr$ be a set, and $\psi$ be a symmetric function such that $\FF{\psi} \ge 0$.
    Then
$$
    \sum_{x,y,z\in A} \psi (x-y) \ov{\psi (x-z)} \psi (y-z)
        \ge
$$
\begin{equation}\label{f:triangles_in_psi}
    \max\left \{
            \frac{1}{|A|^3} \left( \sum_x \psi (x) (A\c A) (x) \right)^3,\,
            |\psi^3 (0)| \cdot |A|,\,
            \frac{1}{|A|^{1/2}} \left( \sum_x |\psi^2 (x)| (A\c A) (x) \right)^{3/2}
        \right\} \,.
\end{equation}
\label{p:triangles_in_psi}
\end{proposition}
\begin{proof}
Put $u(x) = \psi^c (x) = \psi (x)$, $v(x) = A^c (x) \ge 0$.
Let $\{ f_\a \}_{\a\in A}$ be an orthonormal family of the eigenfunctions of the operator
$\ov{\oT}^{N^{-1}\FF{\psi}}_A$ and $\{ \mu_\a \}_{\a\in A}$ be the correspondent nonnegative  eigenvalues.
Then
$$
    \sigma := \sum_{x,y \in A} \psi(x-y) \Cf_3 (v,\ov{u},u) (x,y) = \sum_{\a \in A} \mu_\a d_\a \,,
$$
where by Corollary \ref{c:commutative_C}
\begin{equation}\label{tmp:21.07.2012_1}
    d_\a := \sum_{x,y} \Cf_3 (v,\ov{u},u) (x,y) \ov{f_\a (x)} f_\a (y) = \sum_z v(z) |(f_\a \c u)|^2 (z)
        =
            \sum_{z\in A} |(\psi * f_\a)|^2 (z) \,.
\end{equation}
To get the last identities we have used the arguments from the proof of formula (\ref{tmp:21.07.2012_1'})
and the fact that $\psi=\psi^c$.
Further, because of $f_\a$ is the eigenfunctions of the operator $\ov{\oT}^{N^{-1}\FF{\psi}}_A$, we have
$$
    \mu_\a f_\a (x) = A(x) (\psi * f_\a) (x) \,.
$$
Thus in view of $\| f_\a \|^2_2 = 1$, we obtain $d_\a = \mu^2_\a$.
Note also a trivial lower bound for the largest eigenvalue $\mu_0$, namely
$$
    \mu_0 \ge |A|^{-1} \langle \ov{\oT}^{N^{-1}\FF{\psi}}_A A, A\rangle = |A|^{-1} \sum_x \psi (x) (A\c A) (x) \,.
$$
Hence, applying the last inequality and the assumption $\psi = \psi^c$ once more, we get
$$
    \sigma = \sum_{x,y\in A} \psi (x-y) \Cf_3 (v,\ov{\psi},\psi) (x,y)
        =
            \sum_{x,y,z\in A} \psi (x-y) \ov{\psi(x-z)} \psi(y-z)
                =
                    \sum_{\a\in A} \mu^3_\a
                        \ge
$$
$$
                        \ge
                            \mu^3_0
                                \ge
                                    \frac{1}{|A|^3} \left( \sum_x \psi (x) (A\c A) (x) \right)^3
$$
and the first inequality in (\ref{f:triangles_in_psi})
is proved.
To get the second and the third ones, we use the obtained formula $\sigma = \sum_{\a\in A} \mu^3_\a$,
identities (\ref{F:trace}), (\ref{F:trace_sq}), correspondingly, and H\"{o}lder inequality.
This completes the proof.
$\hfill\Box$
\end{proof}

\bigskip

    Another way to prove (\ref{f:triangles_in_psi}) is to
    write $\Psi(x,y) = \psi(x-y) A(x) A(y)$ as
    $$
        \Psi(x,y) = \sum_{\a,\beta} c_{\a,\beta} \ov{f_\a (x)} f_\beta (y)
    $$
    and note that all terms in the last sum except $\a=\beta$ vanish.
    Further, clearly, $c_{\a,\a} = \mu_\a$.
    Thus, substitution $\Psi(x,y)$ into (\ref{f:triangles_in_psi})
    gives the result.
    In principle, this method gives further generalization of inequality (\ref{f:triangles_in_psi})
    onto larger number of variables in the case of {\it multiplicative subgroups}
    because its eigenfunctions $\chi_\a$ have multiplicative properties
    (see the proof of Proposition \ref{p:convolution_of_eigenvalues}).

    In the general situation we have just the following generalization, where each variable appears twice
$$
    \sum_{x_1,\dots,x_k \in A} \psi (x_1 - x_2) \psi(x_2-x_3) \psi(x_3-x_4) \dots \psi(x_{k-1}-x_k) \psi (x_k-x_1)
        =
            \sum_{\a \in A} \mu^k_\a (\ov{\oT}^{N^{-1} \FF{\psi}}_A)
                \ge
$$
\begin{equation}\label{f:triangles_general}
                \ge
                    \left( \frac{1}{|A|} \sum_{x} \psi(x) (A\c A) (x) \right)^k \,,
\end{equation}
where $k\ge 1$.
Here $\psi$ is a symmetric function and $\FF{\psi} \ge 0$ ($k\ge 3$).
For $k=1$, $k=2$ these are general identities (\ref{F:trace}), (\ref{F:trace_sq}).
If one use the singular--value decomposition lemma for $\Cf_{k+1} (\v{x},y)$, $k\ge 3$
(see section 8 of \cite{ss_E_k}) then some functions $\psi$ in (\ref{f:triangles_general}) can be
replaced
by its moments.
In the case of multiplicative subgroups one can replace $\psi$ in
(\ref{f:triangles_general})
by {\it different} symmetric  $\G$--invariant functions with nonnegative Fourier transform.

    Finally, note also that the condition $\FF{\psi} \ge 0$ is vitally needed in the proposition above.
    Indeed if we consider a dense symmetric subset $Q \subseteq \Gr$ having no solutions of the equation
    $\a+\beta= \gamma$, $\a, \beta, \gamma \in Q$ and put $A=\Gr$,
    $\psi = Q$ then inequality (\ref{f:triangles_in_psi})
    does not hold.
    The phenomenon that such sets must have (large) negative and positive Fourier coefficients
    was considered in \cite{shkr_survey2}, see section 5.


\bigskip

Let $\psi$ be a nonnegative function on an abelian group $\G$,
and $A\subseteq \Gr$ be a set.
Consider the operator $\ov{\oT}^{N^{-1} \FF{\psi}}_A$ and its
orthonormal eigenfunctions $\{ f_j \}_{j\in [|A|]}$.
The condition $\psi \ge 0$ implies that $f_0 \ge 0$, and $\mu_0 \ge 0$.
The next lemma shows that the function $f_0$ is close to $A(x)/|A|^{1/2}$
in some weak sense.

\begin{lemma}
    Let $A\subseteq \Gr$ be a set, and $\psi$ be a nonnegative function,
    $\mu_0$ be the first eigenvalue of the operator  $\ov{\oT}^{N^{-1} \FF{\psi}}_A$.
    Then
\begin{equation}\label{f:g_bound}
        |A| \ge \left( \sum_x f_0 (x) \right)^2
            \ge \max \left\{ \frac{\mu_0}{\| \psi \|_\infty} \,, \frac{\mu^2_0}{\| \psi \|_2^2} \right\} \,,
\end{equation}
    and for the first eigenfunction of  $\ov{\oT}^{N^{-1} \FF{\psi}}_A$,
    $\| f_0 \|_2 = 1$ the following holds
\begin{equation}\label{f:L_infty}
    \| f_0 \|_\infty \le \frac{\| \psi \|_2}{\mu_0} \,.
\end{equation}
If $\FF{\psi} \ge 0$ then
\begin{equation}\label{f:L_infty'}
    \| f_0 \|_\infty \le \frac{\| \psi_1 \|_2}{\mu^{1/2}_0} \,,
\end{equation}
where $\psi = \psi_1 \c \ov{\psi}_1$.
\label{l:g_bound}
\end{lemma}
\begin{proof}
Let $\mu=\mu_0$, $f=f_0$, $g=\sum_x f (x)$.
We have
\begin{equation}\label{tmp:mu_0}
    \mu f (x) = A(x) (\psi * f) (x) \,.
\end{equation}
Thus
\begin{equation}\label{tmp:mu_1}
    \mu = \sum_x f(x) (\psi * f) (x)
\end{equation}
and
\begin{equation}\label{tmp:mu_2}
    \mu^2 = \sum_{x\in A} (\psi * f)^2 (x) \,.
\end{equation}
Formula (\ref{tmp:mu_0}) implies that
$$
    \mu g = \sum_{x\in A} (\psi * f) (x) \,.
$$
Applying Cauchy--Schwarz and (\ref{tmp:mu_2}) (or just Cauchy--Schwarz), we obtain $g^2 \le |A|$.
Further, bound $g^2 \ge \mu \| \psi \|^{-1}_\infty$ easily follows from (\ref{tmp:mu_1}).
Using the formula once more, we get
$$
    \mu \le \sum_x f(x) \cdot \| \psi \|_2 \| f\|_2 =  \| \psi \|_2 g
$$
and we obtain (\ref{f:g_bound}).
Returning to (\ref{tmp:mu_0}) and applying the same argument, we have (\ref{f:L_infty}).
It remains to prove (\ref{f:L_infty'}).
Because of $\FF{\psi} \ge 0$ there is $\psi_1$ such that $\psi = \psi_1 \c \ov{\psi}_1$.
Applying (\ref{tmp:mu_0}) and using Cauchy--Schwarz, we get for any $x\in A$
$$
    \mu |f(x)| \le \sum_{y} (f* \ov{\psi}_1) (x+y) \psi_1 (y)
        \le
            \| \psi_1 \|_2 \cdot ( \sum_{y} |(f* \ov{\psi}_1) (y)|^2 )^{1/2}
                =
                    \| \psi_1 \|_2 \cdot \mu^{1/2} \,,
$$
where formula (\ref{tmp:mu_1}) and the fact $\psi = \psi_1 \c \ov{\psi}_1$ have been used.
This completes the proof.
$\hfill\Box$
\end{proof}

\bigskip

We will use Lemma \ref{l:g_bound} in section \ref{sec:applications2}.

\section{Applications : multiplicative subgroups}
\label{sec:applications}

We begin with an application of  Corollary \ref{cor:energy_weight}.

\begin{theorem}
    Let $p$ be a prime number, and $\Gamma \subseteq \F^*_p$ be a multiplicative subgroup,
    $|\Gamma| = O(p^{2/3})$
    and
    \begin{equation}\label{cond:strange_assumption}
        \E (\Gamma) \le \sqrt{p} |\Gamma|^{\frac{3}{2}} \log |\Gamma| \,.
    \end{equation}
    Then
    \begin{equation}\label{f:subgroups_energy&sumset_w}
        \E (\Gamma) \ll |\Gamma|^{\frac{4}{3}} |\Gamma \pm \Gamma|^{\frac{2}{3}} \log |\Gamma| \,.
    \end{equation}
\label{t:subgroups_energy&sumset}
\end{theorem}
\begin{proof}
Let $Q=\Gamma\pm \Gamma$.
We can assume that
\begin{equation}\label{tmp:14.10.2011_1}
    |Q| = O \left( \frac{\E^{3/2} (\Gamma)}{|\Gamma|^2 \log^{3/2} |\Gamma| } \right)
\end{equation}
because otherwise inequality (\ref{f:subgroups_energy&sumset_w}) is trivial.
Applying formula (\ref{f:energy_wight_1}) of Corollary \ref{cor:energy_weight} with $k=l=1$
and using inequality
$$
    |\G \pm \G_x| \le ((\Gamma \pm \Gamma) \circ ((\Gamma \pm \Gamma))) (x)
$$
(see \cite{kk} or just Proposition \ref{p:characteristic}), we obtain
\begin{equation}\label{tmp:07.09.2011_1-}
    |\Gamma|^2 \E^2 (\Gamma) \le \E_3 (\Gamma) \sum_x (Q\circ Q) (x) (\Gamma \circ \Gamma)^2 (x) \,.
\end{equation}
If we prove that
\begin{equation}\label{tmp:07.09.2011_1}
    \sum_{x\neq 0} (Q \circ Q) (x) (\Gamma \circ \Gamma)^2 (x) \ll \frac{|Q|^{4/3}}{|\Gamma|^{2/3}} |\Gamma|^{7/3} \log |\Gamma|
        \ll
            |Q|^{4/3} |\Gamma|^{5/3} \log |\Gamma|
\end{equation}
then substituting the last formula into (\ref{tmp:07.09.2011_1-})
and using the bound $\E_3 (\Gamma) = O(|\Gamma|^3 \log |\Gamma|)$
from Corollary \ref{cor:E_l},
we get formula (\ref{f:subgroups_energy&sumset_w}).
The term with $x=0$ is $\E_3(\G) |Q| |\G|^2$ and can be handed easily.

From (\ref{tmp:07.09.2011_1-}) it follows that the summation is taken over nonzero $x$ such that
$$
    (Q\circ Q) (x) \ge \frac{\E(\Gamma) |\Gamma|^2}{2\E_3 (\Gamma)} := H \,.
$$
Hence, it is sufficient to prove that
\begin{equation}\label{tmp:07.09.2011_2}
    \sum_{x \neq 0 ~:~ (Q \circ Q) (x) \ge H } (Q \circ Q) (x) (\Gamma \circ \Gamma)^2 (x)
        \ll
            |Q|^{4/3} |\Gamma|^{5/3} \log |\Gamma| \,.
\end{equation}
Let $(Q \circ Q) (\xi_1) \ge (Q \circ Q) (\xi_2) \ge \dots$ and
$(\Gamma \circ \Gamma) (\eta_1) \ge (\Gamma \circ \Gamma) (\eta_2)
\ge \dots$, where nonzero $\xi_1,\xi_2, \dots$ and $\eta_1,\eta_2, \dots$ belong to distinct cosets.
Applying Lemma \ref{l:improved_Konyagin_old} once more, we get
\begin{equation}\label{tmp:07.09.2011_2-}
    (Q \circ Q) (\xi_j) \ll \frac{|Q|^{4/3}}{|\Gamma|^{2/3}} j^{-1/3} \,,
        \quad \mbox{ and } \quad
            (\Gamma \circ \Gamma) (\eta_j) \ll |\Gamma|^{2/3} j^{-1/3} \,,
\end{equation}
provided that  $j|\Gamma| |Q|^2 \ll |\Gamma|^5$ and $j|\Gamma| |Q|^2 |\Gamma| \ll p^3$.
We have $j\ll |Q|^4/ (|\Gamma|^2 H^3)$.
Using inequalities $\E (\Gamma) \ll |\Gamma|^{5/2}$,
$\E_3 (\Gamma) \ll |\Gamma|^3 \log |\Gamma|$, formula (\ref{tmp:14.10.2011_1})
and
assumption (\ref{cond:strange_assumption})
it is easy to check that the last conditions are satisfied.
Applying (\ref{tmp:07.09.2011_2-}), we obtain (\ref{tmp:07.09.2011_1}).
This completes the proof.
$\hfill\Box$
\end{proof}

\bigskip

For example if $|\Gamma| = O(p^{1/2})$ then assumption (\ref{cond:strange_assumption}) holds.
Using trivial lower bound for $\E(\G)$, that is $\E(\G) \ge |\G|^4 / |\G+\G|$, we obtain

\begin{corollary}
    Let $\Gamma \subseteq \F_p^*$ be a multiplicative subgroup, $|\G| \ll \sqrt{p}$.
    Then
    $$
        |\G+\G| \gg \frac{|\G|^{\frac{8}{5}}}{\log^{\frac{3}{5}} |\G|} \,.
    $$
\label{cor:G+G}
\end{corollary}

As for the difference set
 it is known (see \cite{sv})
at the moment
that
$|\Gamma - \Gamma| \gg |\Gamma|^{\frac{5}{3}} \log^{-\frac{1}{2}} |\G|$
for an arbitrary multiplicative subgroup $\Gamma$ with  $|\Gamma| \ll \sqrt{p}$.
We will see soon that the condition $|\G| \ll \sqrt{p}$ in Corollary \ref{cor:G+G}
can be relaxed (see Theorem \ref{t:subgroup_energy} below).

\bigskip

\begin{corollary}
    Let $\Gamma \subseteq \F_p^*$ be a multiplicative subgroup, $-1\in \Gamma$
    such that $|\Gamma| \ge p^{\kappa}$, where
     $\kappa > \frac{33}{68}.$
    Then for all sufficiently large $p$ we have $6\Gamma = \F_p$.
    If $\kappa > \frac{55}{112}$ then $\F^*_p \subseteq 6\Gamma$ without condition $-1\in \Gamma$.
\label{c:6R_subgroups}
\end{corollary}
\begin{proof}
Put $S = \Gamma+\Gamma$, $n=|\Gamma|$, $m = |S|$, and $\rho =
\max_{\xi \neq 0} |\FF{\Gamma} (\xi)|$. By a well--known upper bound
for Fourier coefficients of multiplicative subgroups (see e.g.
Corollary 2.5 from \cite{ss} or Lemma \ref{l:G-inv_bound_F})
we have $\rho \le p^{1/8} \E^{1/4}
(\Gamma)$. If $\F_p^* \not\subseteq 6\Gamma$ then for some
$\lambda\neq 0,$ we
obtain
$$
    0 = \sum_{\xi} \FF{S}^2 (\xi) \FF{\Gamma}^2 (\xi) \FF{\lambda \Gamma} (\xi)
        =
            m^2n^3+\sum_{\xi\neq 0} \FF{S}^2 (\xi) \FF{\Gamma}^2 (\xi) \FF{\lambda \Gamma} (\xi) \,.
$$
Therefore, by the estimate $\rho \le p^{1/8} \E^{1/4} (\Gamma)$
and Parseval identity we get
\begin{equation}\label{f:cor_6R_subgroups_main}
    n^3 m^2\le \rho^3mp\ll (p^{1/8} \E^{1/4})^3 mp \,.
\end{equation}
Now applying
formula (\ref{f:subgroups_energy&sumset_w})
and
$m \gg n^{5/3} \log^{-1/2} n$ (see \cite{sv}),
we obtain the required result.
To obtain the same without condition $-1\in \Gamma$ just use formula (\ref{f:cor_6R_subgroups_main}),
combining with formula (\ref{f:subgroups_energy&sumset_w})
and
apply the lower bound for $\G+\G$
from Corollary \ref{cor:G+G}.
$\hfill\Box$
\end{proof}

\bigskip

\begin{remark}
The inclusion $\F^*_p \subseteq 6\Gamma$ was obtained in
\cite{ss_E_k} under the
assumption $\kappa > \frac{99}{203}$.
Even more stronger results than containing in Corollary \ref{c:6R_subgroups}
were obtained by A. Efremov using  further  development of our method
(unpublished).
\end{remark}

\bigskip

Now we obtain a result about the additive energy of multiplicative subgroups.

\begin{theorem}
    Let $p$ be a prime number and $\G \subseteq \F^*_p$ be a multiplicative subgroup.
    Then
    \begin{equation}\label{f:subgroup_energy}
        \E (\G)
            \ll
                \max\{ |\G|^{\frac{22}{9}} \log^{} |\G|, |\G|^3 p^{-\frac{1}{3}} \log^{\frac{4}{3}} |\G| \} \,.
    \end{equation}
    More precisely,
    \begin{equation}\label{f:subgroup_energy_restriction}
        \E (\G) \ll |\G|^{\frac{22}{9}} \log^{\frac{2}{3}} |\G|
    \end{equation}
    provided by
    $|\G| \ll p^{\frac{3}{5}} \log^{-\frac{6}{5}} p$.
    Moreover, if $|\G| < \sqrt{p}$, and $k\ge 2$ then we have
    \begin{equation}\label{f:subgroup_T_k_new}
        \T_k (\G) \ll_k 
                        |\G|^{2k - \frac{17}{9} + \frac{16}{3} 2^{-2k}} \log^{\frac{2}{3}} |\G|
                        \,.
    \end{equation}
\label{t:subgroup_energy}
\end{theorem}
\begin{proof}
Let $|\G|=t$, $\E_3 (\G) = \E_3$, $\E(\G) = \E = t^3 / K$, $K\ge 1$, $\T_l = \T_l (\G)$, $l\ge 2$.
We need to find the lower bound for $K$ and the upper bound for $\T_k$.
Put
$$
    \sigma_* = \sum_{x\in \G} (\G*(\G \c \G))^2 (x) \,.
$$
By Cauchy--Schwarz
$$
    \sigma_* \ge \frac{\E^2}{t} = \frac{t^5}{K^2}
$$
(actually in the case of multiplicative subgroups equality holds).
Applying formula (\ref{f:C_3_lower_bound}) of Proposition \ref{p:eigenfunctions_Gamma}
with $\psi (x) = u(x) = (\G\c \G) (x)$, $v(x) = \G(x)$ and the coset $-\G$, we obtain
$$
    \sum_{x,y,z\in \G} \psi (y-x) \psi (z-x) \psi (y-z)
        \ge
            \frac{\E}{t^2} \cdot \sigma_* \,.
$$
In other words
\begin{equation}\label{tmp:19.07.2012}
    \sum_{\a,\beta} \psi (\a) \psi (\beta) \psi (\a-\beta) \Cf_3 (\G) (\a,\beta)
        \ge
            \frac{\E}{t^2} \cdot \sigma_* \,.
\end{equation}
Clearly,
\begin{equation}\label{tmp:19.07.2012_1}
    \sum_{\a \neq 0, \beta \neq 0, \a \neq \beta} \psi (\a) \psi (\beta) \psi (\a-\beta) \Cf_3 (\G) (\a,\beta)
        \ge
            2^{-1} \frac{\E}{t^2} \cdot \sigma_*
\end{equation}
because
if $\a,\beta$ or $\a-\beta$ equals zero then
$$
    t \E_3 (\G) \gg \frac{t^6}{K^3}
$$
which implies $K\gg t^{2/3} \log^{-1/3} t$ and the result follows.
Further the summation in (\ref{tmp:19.07.2012_1}) can be taken over nonzero $\a$ such that
\begin{equation}\label{tmp:19.07.2012_1*}
    \psi (\a) \ge 2^{-4} \frac{\E}{t^2} := d
\end{equation}
because of for other $\a$, we have
$$
    3d \sigma_* < 2^{-1} \frac{\E}{t^2} \cdot \sigma_*
$$
with contradiction.
In the last formula we have use the fact that $\G$ is a subgroup.
Thus suppose that formula
\begin{equation}\label{tmp:19.07.2012_1}
    \sum_{\a \neq 0, \beta \neq 0, \a \neq \beta ~:~ \psi (\a), \psi (\beta), \psi (\a-\beta) \gg d}
             \psi (\a) \psi (\beta) \psi (\a-\beta) \Cf_3 (\G) (\a,\beta)
        \ge
            2^{-2} \frac{\E}{t^2} \cdot \sigma_*
\end{equation}
takes place, where $d$ is defined by (\ref{tmp:19.07.2012_1*}).
By one more application of the Cauchy--Schwarz,
we obtain
\begin{equation}\label{f:energy_start}
    \sum_{\a \neq 0, \beta \neq 0, \a \neq \beta ~:~ \psi (\a), \psi (\beta), \psi (\a-\beta) \gg d}\,
            \psi^2 (\a) \psi^2 (\beta) \psi^2 (\a-\beta)
        \gg
            \frac{\E^2}{t^4} \cdot \sigma^2_* \E^{-1}_3 \gg \frac{\E^6}{t^6 \E_3}\,.
\end{equation}
Put
$$
    S_i = \{ x\in \G-\G\,, x\neq 0 ~:~ 2^{i-1} d < \psi (x) \le 2^i d \} \,, \quad i \in [l] \,, \quad l \ll \log t \,.
$$
Then
\begin{equation}\label{tmp:19.07.2012_2}
    d^6 \cdot \sum_{i,j,k = 1}^l 2^{2i+2j+2k} \sum_\a S_i (\a) (S_j * S_k) (\a)
        \gg
            \frac{\E^2}{t^4} \cdot \sigma^2_* \E^{-1}_3 \,.
\end{equation}
To estimate the inner sum in (\ref{tmp:19.07.2012_2}) we use Lemma \ref{l:improved_Konyagin_old}.
Suppose that for all $i,j,k\in [l]$ the following
two inequalities hold
\begin{equation}\label{cond:S_ijk_1}
    |S_i| |S_j| |S_k| \ll t^5
\end{equation}
and
\begin{equation}\label{cond:S_ijk_2}
    |S_i| |S_j| |S_k| t \ll p^3 \,.
\end{equation}
Then by Lemma \ref{l:improved_Konyagin_old}
$$
    d^6 t^{-1/3} \cdot \sum_{i,j,k = 1}^l 2^{2i+2j+2k} (|S_i| |S_j| |S_k|)^{2/3}
        \gg
            \frac{\E^2}{t^4} \cdot \sigma^2_* \E^{-1}_3 \,.
$$
We can suppose that $K \ll t^{5/9} \log^{-2/3} t$ because otherwise the result is trivial.
Note also a trivial upper bound for the size of any $S_i$, namely
\begin{equation}\label{tmp:19.07.2012_4}
    2^{i-1} d |S_i| \le \sum_{x \neq 0} \psi (x) \le t^2 \,.
\end{equation}
or in other words
$$
    |S_i| \ll 2^{-i} K t \ll K t \,.
$$
In particular
\begin{equation}\label{tmp:19.07.2012_3}
    t^3 |S_i| \ll t^4 K \ll t^4 t^{5/9} \log^{-2/3} t \ll p^3
\end{equation}
because of $t\ll p^{27/41}$.
In view of (\ref{tmp:19.07.2012_3}), a trivial inequality $|S_i| t^2 \ll t^5$,
 and Lemma \ref{l:improved_Konyagin_old},
we obtain
\begin{equation}\label{f:S_i_main}
    |S_i| \ll \frac{t^3}{2^{3i} d^3}  \,.
\end{equation}
A little bit worse bound
\begin{equation}\label{f:S_i_main'}
    |S_i| \ll \frac{t^3 \log t}{2^{3i} d^3}
\end{equation}
but for all $t\ll p^{2/3}$
follows from the estimate of $\E_3$, see Corollary \ref{cor:E_l}.
Substituting
(\ref{f:S_i_main})
into (\ref{tmp:19.07.2012_2}) gives us
$$
    t^{6-1/3} \log^3 t \gg \frac{\E^2}{t^4} \cdot \sigma^2_* \E^{-1}_3
$$
and after some calculations we obtain $K \gg t^{5/9} \log^{-2/3} t$.
It is remain
to check (\ref{cond:S_ijk_1}), (\ref{cond:S_ijk_2}).
Applying (\ref{tmp:19.07.2012_4}) and $K \ll t^{5/9} \log^{-2/3} t$, we have
\begin{equation}\label{tmp:23.08.2012_*}
    |S_i| |S_j| |S_k| \ll  (K t)^3 \cdot 2^{-(i+j+k)}
        \ll (K t)^3 \ll t^{14/3} \log^{-2} t \ll t^5
\end{equation}
and inequality (\ref{cond:S_ijk_1}) holds.
Finally
\begin{equation}\label{tmp:prod_p}
    |S_i| |S_j| |S_k| t \ll t^{17/3} \log^{-2} t \cdot 2^{-(i+j+k)} \ll t^{17/3} \log^{-2} t \ll p^3
\end{equation}
provided by
$t \ll p^{\frac{9}{17}}$ and $K \ll t^{5/9} \log^{-2/3} t$.

Now let us prove the same for larger $t$.
Returning to (\ref{tmp:19.07.2012_2}),
applying the first bound from estimate (\ref{f:G-inv_bound_F}) of Lemma \ref{l:G-inv_bound_F}
and using Fourier transform, we obtain
\begin{equation}\label{tmp:22.08.2012_1}
    \sum_\a S_i (\a) (S_j * S_k) (\a) \ll \max\{ p^{-1} |S_i| |S_j| |S_k|, \sqrt{p/t} (|S_i| |S_j| |S_k|)^{1/2} \} \,.
\end{equation}
We have used the first formula of Lemma \ref{l:G-inv_bound_F}
it is the most effective in the choice of parameters.
If the maximum from (\ref{tmp:22.08.2012_1}) is attained on the first term then by (\ref{tmp:19.07.2012_2}),
and trivial inequality
\begin{equation}\label{tmp:04.11.2012_2}
    |S_j| d^2 2^{2j} \le \E \,,
\end{equation}
we get
\begin{equation}\label{tmp:22.08.2012_2'}
    \E \ll \frac{t^3 \log^{4/3} t}{p^{1/3}} \,,
\end{equation}
and if it is attained on the second term, we have by (\ref{tmp:04.11.2012_2})
\begin{equation}\label{tmp:22.08.2012_2''}
    2^{i+j+k} \gg \frac{\E^{3/2} t^{1/2} }{ p^{1/2} \E_3 \log^3 t } \,.
\end{equation}
Simple computations show that having (\ref{tmp:22.08.2012_2'}) we easily get (\ref{f:subgroup_energy})
for $t\ll p^{3/5} \log^{-6/5} t \ll p^{21/47}$.
Further by (\ref{tmp:04.11.2012_2}) we have an analog of (\ref{tmp:prod_p})
\begin{equation}\label{tmp:06.11.2012_1}
    |S_i| |S_j| |S_k| t
        \ll
            \frac{\E^3}{d^6} t 2^{-2(i+j+k)}
                \ll t^{17/3}  2^{-2(i+j+k)} \log^{-2} t \ll p^3
\end{equation}
%
%
%
%
Thus substitution (\ref{tmp:22.08.2012_2''}) into (\ref{tmp:06.11.2012_1})
gives $t\ll p^{3/5} \log^{-6/5} t$.
This completes the proof of inequality (\ref{f:subgroup_energy_restriction}).
Bound (\ref{f:subgroup_energy}) is obtained by accurate calculations using
inequality (\ref{f:S_i_main'})
in the wide range $t\ll p^{2/3}$ and estimate (\ref{tmp:22.08.2012_2'}).

\bigskip

To get (\ref{f:subgroup_T_k_new}) take $\psi (x) = ((\G *_{k-1} \G) \c (\G *_{k-1} \G)) (x)$
and use previous arguments.
We have
$$
    \sum_{\a,\beta} \psi (\a) \psi (\beta) \psi (\a-\beta) \Cf_3 (\G) (\a,\beta)
        \ge
            t^{-3} \T^3_{k+1} (\G)
$$
and if $\a$, $\beta$ or $\a-\beta$ equals zero then by Theorem \ref{t:3_d_moment}, we get
$$
    t^{6k-4+3^{-1} (1+2^{3-2k}) + 2^{1-k} } \cdot \log t \gg_k \T_k (\G) \cdot \sum_{x} \psi^2 (x) (\G \c \G) (x)
            \gg
        t^{-3} \T^3_{k+1} (\G)
$$
and the result follows.
As above
$$
    \sum_{\a \neq 0, \beta \neq 0, \a \neq \beta ~:~ \psi (\a), \psi (\beta), \psi (\a-\beta) \gg d}\,
            \psi^2 (\a) \psi^2 (\beta) \psi^2 (\a-\beta)
        \gg
            t^{-6} \T^6_{k+1} (\G) \cdot \E^{-1}_3 \,,
$$
where
$$
    d = \frac{\T^3_{k+1}}{t^3 \T_{2k} \E^{1/2}_3} \,.
$$
Consider the sets $S_i$ similar way, we obtain by Theorem \ref{t:3_d_moment} that
$|S_i| \ll t^{6k-4+2^{2-2k}} / 2^{3i} d^3$
and hence
\begin{equation}\label{tmp:23.07.2012_1}
    t^{12k-8+ 2^{3-2k} - 1/3} \log^3 t \gg t^{-6} \T^6_{k+1} (\G) \cdot \E^{-1}_3 \,,
\end{equation}
provided by inequalities (\ref{cond:S_ijk_1}), (\ref{cond:S_ijk_2}) hold.
Inequality (\ref{tmp:23.07.2012_1}) implies that
$$
    \T_{k+1} \ll t^{2k+1/9+2^{2-2k}/3} \log^{2/3} t
$$
and we are done.
Using Theorem \ref{t:3_d_moment} it is easy to check that
(\ref{cond:S_ijk_1}) takes place.
Hence, because of $t< \sqrt{p}$  inequality (\ref{cond:S_ijk_2}) holds automatically.
This completes the proof.
$\hfill\Box$
\end{proof}

\bigskip

Thus, inequality (\ref{f:subgroup_T_k_new}) is better then Theorem \ref{t:3_d_moment} for
$k=2$ and for $k=3$, namely, $\T_3 (\G) \ll t^{151/36} \log^{2/3} t$.
Using more accurate  arguments from \cite{K_Tula}
one can, certainly,  improve our bounds for large $k$.
We do not make such calculations.

Note, finally, that
inequality (\ref{f:subgroup_energy_restriction}) gives bounds for $\E(\G)$
which are
better than  Theorem \ref{t:3_d_moment}  if
$|\G| \ll p^{\frac{2}{3}} \log^{-\frac{8}{3}} p$.

\bigskip

Now we formulate Corollary 39 from \cite{ss_E_k},
which was obtained by eigenvalues method of section \ref{sec:eigenvalues} also.

\begin{corollary}
    Let $p$ be a prime number,
    $\Gamma_* \subseteq \f^*_q$
    be a coset of a multiplicative subgroup $\Gamma$.
    If $Q^{(y)} \subseteq Q^k,\, y\in \Gamma'$ is an arbitrary family
    of sets, then
    $$
        \Big| \bigcup_{y\in \Gamma'} (Q^{(y)} \pm \Delta (y)) \Big|
            \ge \frac{|\Gamma|}{|\Gamma'| \E_{k+1} (\Gamma_*,Q)} \cdot \Big( \sum_{y\in \Gamma'} |Q^{(y)}| \Big)^2 \,.
    $$
    In particular for every set $A \subseteq \Gamma_*$, and every $\Gamma$--invariant set $Q$, we have
    \begin{equation}\label{f:subgroups_pred1}
        |Q+A| \ge |A| \cdot \frac{|\Gamma| |Q|^2}{\E_2 (\Gamma_*,Q)} \,.
    \end{equation}
\label{c:subgroups}
\end{corollary}

Corollary above combining with Theorem \ref{t:subgroup_energy}
say that multiplicative subgroups have strong expanding property.

\begin{corollary}
    Let $p$ be a prime number,
    $\Gamma_* \subseteq \f^*_q$ be a coset of a multiplicative subgroup $\Gamma$, $|\G| \ll p^{\frac{6}{11}}$.
    Then for any  $A \subseteq \Gamma_*$, we have
    $$
        |A+\G| \gg \frac{|A| |\G|^{5/9}}{\log^{2/3} |\G|} \,.
    $$
\label{cor:expanding_strong}
\end{corollary}

Ordinary application of Cauchy---Schwarz gives $|A+\G| \gg |A|^{1/2} |\G|^{7/9} \log^{-1/3} |\G|$
for any set $A$ and any multiplicative subgroup $\G$, $|\G| \ll \sqrt{p}$.

Theorem \ref{t:subgroup_energy} gives a direct application to the exponential sums over subgroups.

\begin{corollary}
    Let $p$ be a prime number, $\G$ be a multiplicative subgroup, $|\G| \ll p^{\frac{6}{11}}$.
    Then
\begin{equation}\label{f:exp_sums_subgr}
    \max_{\xi \neq 0} |\FF{\G} (\xi)| \ll \min\{ p^{1/8} |\G|^{11/18}, p^{1/4} |\G|^{13/36} \} \cdot \log^{1/6} |\G| \,.
\end{equation}
\label{cor:exp_sums_subgr}
\end{corollary}
\begin{proof}
    Let $\rho = \max_{\xi \neq 0} |\FF{\G} (\xi)|$.
    Because of $\rho \le p^{1/8} \E^{1/4} (\G)$ and
    $\rho \le p^{1/4} |\G|^{-1/4} \E^{1/4} (\G)$
    (see e.g. Corollary 2.5 from \cite{ss} or Lemma \ref{l:G-inv_bound_F}),
    applying Theorem \ref{t:subgroup_energy},
    we obtain (\ref{f:exp_sums_subgr}).
    This completes the proof.
$\hfill\Box$
\end{proof}

\bigskip

For any function $f:\G \to \mathbb{C}$ by $\T^{\m}_k (f)$ denote the quantity
$$
    \T^{\m}_k (f)
        =
            \sum_{x_1,\dots,x_k,x'_1,\dots,x'_k ~:~ x_1\dots x_k = x'_1 \dots x'_k} f(x_1) \dots f(x_k) \ov{f(x'_1)} \dots \ov{f(x'_k)} \,.
$$
$\T^{\m}_k (f)$ is a multiplicative analog of $\T_k (f)$ from section \ref{sec:definitions}.
Write also $\E^\m$ for $\T^\m_2$.

Using the eigenvalues method, we want to find
some relations between $\T^{\m}_k (A)$ and another characteristics of
an arbitrary subset $A$ of a multiplicative subgroup.
We need in a simple lemma.

\begin{lemma}
\label{l:T_k_in_subgroups}
    Let $\Gamma\subseteq \f^*_q$ be a multiplicative subgroup.
    Suppose that $f(x) = \sum_\a c_\a \chi_\a(x)$ is an arbitrary function
    with support on $\G$.
    Then
    $$
        \T^{\m}_k (f) = |\G|^{k-1} \sum_\a |c_\a|^{2k} \,.
    $$
\end{lemma}
\begin{proof}
    By the multiplicative property of the functions $\chi_\a (x)$, we have
 $$
        \sum_\a |c_\a|^{2k} = \sum_\a |\sum_x f(x) \chi_\a (x)|^{2k} =
 $$
$$
        \sum_\a \sum_{x_1,\dots,x_k,x'_1,\dots,x'_k}
            f(x_1) \dots f(x_k) \ov{f(x'_1)} \dots \ov{f(x'_k)} \chi_\a (x_1) \dots \chi_\a (x_k) \ov{\chi_\a (x'_1)} \dots \ov{\chi_\a (x'_k)}
        = \frac{\T^\m_k (P)}{|\G|^{k-1}}
$$
as required.
$\hfill\Box$
\end{proof}

\begin{corollary}
    Let $\Gamma\subseteq \f^*_q$ be a multiplicative subgroup, and $A\subseteq \G$.
    Then
    $
        \T^{\m}_k (A) \ge \frac{|A|^{2k}}{|\G|} \,.
    $
\label{cor:T_k_below_G}
\end{corollary}

\bigskip


Now formulate a result on
a relation between $\T^{\m}_k (A)$ and some another characteristics of
an arbitrary subset $A$ of a multiplicative subgroup.

\begin{proposition}
    Let
    $\Gamma\subseteq \f^*_q$ be a multiplicative subgroup,
    and
    $A$ be any subset of $\G$.
    Then for an arbitrary integer $k\ge 2$, we have
    \begin{equation}\label{f:P_variance}
        |A|^2 \le  |\G|^2 (\T^{\m}_k (A))^{1/k}
            \cdot \min_h \left( \frac{\| h\|_1}{\| h \|_2} \right)^{2/k} \frac{\sum_x |h(x)|^2}{\sum_x (h \c \ov{h}) (x) (\G \c \G) (x)} \,,
    \end{equation}
    where the minimum is taken over all nonzero $\G$--invariant functions.
    In the case $k=2$, we also have
    \begin{equation}\label{f:P_variance_k=2}
        |A| \le |\G|^{1/4} (\E^{\m} (A))^{1/4}
    \end{equation}
    and
    \begin{equation}\label{f:P_variance_k=2'}
        \E_l (A,\G) \le |\G|^{-1/2} \E_{2l-1}^{1/2} (\G) (\E^{\m} (A))^{1/2}
    \end{equation}
    for any $l\ge 2$.
\label{p:Vinogradov}
\end{proposition}
\begin{proof}
    Take $g (x) = (h \c \ov{h}) (x)$.
    Then $\FF{g} \ge 0$.
    Now proceed as in the proof of formula (\ref{f:connected}) from Proposition \ref{p:eigenfunctions_Gamma}.
    Let $A = \sum_{\a} c_\a \chi_\a$ and $\mu_\a = \mu_\a (\ov{\oT}^{q^{-1}\FF{g}}_\G)$.
    By H\"{o}lder, we have
    \begin{equation}\label{f:E(G,P)}
        \sum_x g (x) (A\c A)(x) = \sum_\a |c_\a|^2 \mu_\a
            \le
                \left( \sum_\a |c_\a|^{2k} \right)^{1/k} \left( \sum_\a \mu^{\frac{k}{k-1}}_\a \right)^{1-1/k} \,.
    \end{equation}
    Applying Lemma \ref{l:T_k_in_subgroups}, we get
    \begin{equation}\label{tmp:07.11.2011_1}
        \sum_\a |c_\a|^{2k} = \frac{\T^\m_k (A)}{|\G|^{k-1}} \,.
    \end{equation}
    On the other hand
    \begin{equation}\label{tmp:07.11.2011_2}
        \left( \sum_\a \mu^{\frac{k}{k-1}}_\a \right)^{1-1/k}
            \le
                \mu^{1/k}_0 \cdot ( \sum_\a \mu_\a )^{1-1/k}
                    \le
                        \| h \|_1^{2/k} \| h \|_2^{2-2/k} |\G|^{1-1/k}  \,,
    \end{equation}
    where a trivial estimate
    $$
        \mu_0 = |\G|^{-1} \sum_x g (x) (\G\c \G)(x) \le ( \sum_x |h(x)| )^2
    $$
    and
    a particular case of formula (\ref{F:trace}), namely,
    $$
        \sum_\a \mu_\a = |\G| g (0) = |\G| \| h \|_2^2
    $$
    were
    used.
    Substituting (\ref{tmp:07.11.2011_1}) and (\ref{tmp:07.11.2011_2}) into (\ref{f:E(G,P)}),
    we get
    \begin{equation}\label{tmp:08.11.2011_4}
        \frac{|A|^2}{|\G|^{2}} \sum_x g (x) (\G\c \G)(x) = c^2_0 \mu_0 \le \sum_x g (x) (A\c A)(x)
            \le
                \| h \|_2^2 (\T^{\m}_k (A))^{1/k} \left( \frac{\| h\|_1}{\| h \|_2} \right)^{2/k}
    \end{equation}
    and (\ref{f:P_variance}) is proved.

    To obtain (\ref{f:P_variance_k=2}), we just note that in the case $k=2$ the sum
    $\sum_\a \mu^{2}_\a$ from (\ref{tmp:07.11.2011_2}) can be computed.
    Indeed by formula (\ref{F:trace_sq})
    \begin{equation}\label{tmp:08.11.2011_5}
        \sum_\a \mu^{2}_\a = \sum_x |g (x)|^2 (\G \c \G) (x) = \sum_x |(h \c \ov{h}) (x)|^2 (\G \c \G) (x)
    \end{equation}
    and after using the same arguments as above, we have
    \begin{equation}\label{tmp:08.11.2011_3}
        |A|^2 \le |\G|^{3/2} (\E^{\m} (A))^{1/2}
            \cdot \min_h \frac{\left( \sum_x |(h \c \ov{h}) (x)|^2 (\G \c \G) (x) \right)^{1/2}}
                              {\sum_x (h \c \ov{h}) (x) (\G \c \G) (x)} \,.
    \end{equation}
    Optimizing the last inequality over $h$ (taking $h(x) \equiv 1$), we obtain
    (\ref{f:P_variance_k=2}).
    To get
    (\ref{f:P_variance_k=2'})
    take  $g(x) = (\G \c \G)^{l-1} (x)$, use formula (\ref{tmp:08.11.2011_5})
    and repeat  the arguments from (\ref{f:E(G,P)}), (\ref{tmp:08.11.2011_5}).
    After some computations, we have
    $$
        \sum_x g (x) (A\c A)(x) = \E_l (A,\G) \le |\G|^{-1/2} \E^{1/2}_{2l-1} (\G) (\E^{\m} (A))^{1/2}
    $$
    as required.
    This completes the proof of the proposition.
$\hfill\Box$
\end{proof}



Note that formula (\ref{f:P_variance_k=2}) is just reformulation of Lemma \ref{l:T_k_in_subgroups}.
Formulas (\ref{f:P_variance})--(\ref{f:P_variance_k=2'}) give an explanation
why $\G$ is a eigenfunction of operator $\oT^{\FF{g}}_\G$.
The thing is $\T^\m_k (\G)$ is maximal over all subsets of a multiplicative subgroup.


Below we will deal with
the field $\f_p$,
where $p$ is a prime number.
There are plenty results about the quantity $\T^\m_k$ for arithmetic progressions in $\f_p$.

\begin{theorem}
\label{t:CSZ}
    $1)~$ Let $P\subseteq \f^*_p$ be an arithmetic progression.
    Then \cite{CSZ}
    $$
        \T_2^\m (P) = \frac{|P|^4}{p} + O(|P|^{2+o(1)}) \,.
    $$
    $2)~$ If $|P| \ll p^{1/8}$ then \cite{CG} the number of solutions of the congruence
    $$
        xyz \equiv \lambda \pmod p\,,  \quad \lambda \neq 0\,, \quad x,y,z \in P
    $$
    does not exceed $|P|^{o(1)}$ (uniformly over $\lambda$).\\
    $3)~$ If $\nu$ is a positive integer,
            $|P| \ll p^{c(\nu)}$, where $c(\nu)>0$ is some constant depends on $\nu$ only.
            Then \cite{BKS_oracle} the number of solutions of the congruence
    $$
        x_1 \dots x_\nu \equiv \lambda \pmod p\,,  \quad \lambda \neq 0\,, \quad x_1 \dots x_\nu \in P
    $$
    is bounded by $$ \exp\l( c'(\nu) \frac{\log |P|}{\log \log |P|} \r) \,,$$
    where $c'(\nu) > 0$ depends on $\nu$ only.
\end{theorem}

\begin{corollary}
\label{cor:E_3_P}
    Let $\Gamma\subseteq \f^*_q$ be a nontrivial multiplicative subgroup.
    Then for any progression $P\subseteq \G$ the following holds
    \begin{equation}\label{f:P_size_weak}
        |P| \ll |\G|^{1/2+o(1)} \,,
    \end{equation}
    Suppose that $|\G| \ll p^{2/3}$ and $l\ge 3$.
    Then
    \begin{equation}\label{f:d(G),k=2_P+G}
        \E(P,\G) \ll |P|^{1+o(1)} |\G| \log^{1/2} |\G|
            \quad \mbox{ and } \quad
            \E_l (P,\G) \ll |P|^{1+o(1)} |\G|^{l-1} \,.
    \end{equation}
\end{corollary}
\begin{proof}
Suppose that $P \subseteq \G$ is an arbitrary progression.
By Theorem \ref{t:CSZ}, we have
\begin{equation}\label{tmp:08.08.2012_1}
    \E^\m (P) = \frac{|P|^4}{p} + O(|P|^{2+o(1)}) \,.
\end{equation}
If the first term is dominated then applying  (\ref{f:P_variance_k=2}), we get
$$
    |P| \le \frac{2^{1/4}|P|}{p^{1/4}} |\G|^{1/4}
$$
with contradiction.
Thus the second term in (\ref{tmp:08.08.2012_1})
is dominated and using (\ref{f:P_variance_k=2}), we obtain (\ref{f:P_size_weak}).
Applying Theorem \ref{t:CSZ} once again, formula
(\ref{f:P_variance_k=2'})
and Corollary \ref{cor:E_l}, we get (\ref{f:d(G),k=2_P+G}).
This completes the proof.
%
%
%
$\hfill\Box$
\end{proof}



Clearly, the condition $|\G| \ll p^{2/3}$ can be relaxed for large $l$.
Obviously,
inequality
(\ref{f:d(G),k=2_P+G}) is the best possible up to $|P|^{o(1)}$ factor.

\begin{remark}
    The arguments from the proof of Proposition \ref{p:Vinogradov} give
    (we consider the simplest case $l=2$)
    the following asymptotic formula
    $$
        \E (P,\G) = \sum_{x} (\G \c \G) (x) (P\c P) (x)
            =
                \frac{|P|^2 \E(\G)}{|\G|^2} + \theta |P|^{1+o(1)} |\G|^{-1/2} (\E^*_3 (\G))^{1/2} \,,
    $$
    where $|\theta| \le 1$ and $\E^*_3 (\G) = \sum_{\a\neq 0} \mu^2_\a$.
    Here $P \subseteq \G$ is an arithmetic
    progression.
    The asymptotic formula works just for large subgroups
    of size $p^{1-\delta}$, $\delta>0$.
\end{remark}

\begin{remark}
    Certainly, inequality
    \begin{equation}\label{f:d(G),k=2_P+G'}
        |P+\G| \gg |\G| |P|^{1-o(1)} \log^{-1/2} |\G|
    \end{equation}
    follows from (\ref{f:d(G),k=2_P+G}) by Cauchy--Schwartz
    and
    one can obtain analog of formula
    (\ref{f:d(G),k=2_P+G'})
    for $l$ larger  than two, namely,
    $|\G^{l-1} + \Delta_{l-1} (P)| \gg |P|^{1-o(1)} |\G|^{l-1}$.
    Nevertheless in the case $l>2$
    a more exact and general bound
    was obtained in \cite{ss_E_k} (see Corollary 39, the case $k\ge 2$), namely,
    \begin{equation}\label{f:expanding_known}
        |\G^{2} + \Delta_{2} (A)| \gg |A| |\G|^{2} \log^{-1} |\G|
            \quad \mbox{ and } \quad |\G^{l-1} + \Delta_{l-1} (A)| \gg |A| |\G|^{l-1} \,, ~~~ l>3
    \end{equation}
    for {\it any} $A\subseteq \G$.
\end{remark}

Finally, for the sake of completeness and because of it is difficult to find
in the literature,
we add a very simple result on progressions in small subgroups.

\begin{proposition}
\label{st:Vinogradov_real}
    Let $p$ be a prime number, $\delta \in (0,1)$ is a real number.
    Suppose that $\Gamma\subseteq \f^*_p$ is a multiplicative subgroup, $|\G| = p^{1-\d}$,
    and $P = \{ a,2a,\dots, sa \} \subseteq \G$, $a\neq 0$.
    Then there is an absolute constant $C>0$ such that for all $p\ge p_0(\d)$, we have
    \begin{equation}\label{f:Vinogradov_real}
        |P| \le \exp( C \sqrt{\d^{-1} \log (1/\d) \log p} ) \,.
    \end{equation}
    Moreover for any
    such
    arithmetic progression $P$,
    $\log |P| \gg \sqrt{\d^{-1} \log (1/\d) \log p}$
    the following holds
    \begin{equation}\label{f:Vinogradov_real_sum}
        |P\bigcap \G| \le |P|^{1-\d/4} \,.
    \end{equation}
\end{proposition}
\begin{proof}
Suppose for a moment that $P = \{ 1,2,\dots,s \} \subseteq \G$.
If $\log |P| \ll \sqrt{\d^{-1} \log (1/\d) \log p}$ then it is nothing to prove.
On the other hand we can take $s\ge 1$ as small as we want.
Thus suppose that $\log s \sim \sqrt{\d^{-1} \log (1/\d) \log p}$.

Because of we take $p\ge p_0(\d)$ sufficiently large
we can choose
minimal
$k\ge 2$ such that
$k \ge \log p / \log s$.
One can quickly
check that $k\ll \log s$.
Using Dirichlet's method (see \cite{IK}) it is easy to prove
\begin{equation}\label{f:T_k_for_progressions}
        \T^\m_k (P) \le |P|^k \left( \frac{C\log |P|}{k} \right)^{k(k-1)} \,,
\end{equation}
where $C>0$ is an absolute constant.
By Corollary \ref{cor:T_k_below_G} and formula (\ref{f:T_k_for_progressions}), we have
$$
    \frac{s^{2k}}{|\G|} \le \T^{\m}_k (P) \le s^k \left( \frac{C\log s}{k} \right)^{k^2} \,.
$$
In other words
$$
    \log s \le k \log (C k^{-1} \log s) + k^{-1} \log |\G| \,.
$$
Hence
$$
    \d \log s \ll k \log (C k^{-1} \log s) \ll \frac{\log p}{\log s} \cdot \log (C \log^2 s \cdot \log^{-1} p) \,.
$$
Put $x=\log^2 s \cdot \log^{-1} p$.
Then the last inequality can be rewritten as $x \ll \d^{-1} \log Cx$.
In other words $x\ll \d^{-1} \log (1/\d)$ and we have formula (\ref{f:Vinogradov_real}) because of our
method equally works for progressions of the form  $\{ a,2a,\dots, sa \}$
as well.

Thanks to Lemma \ref{l:T_k_in_subgroups} we can obtain estimate
(\ref{f:Vinogradov_real_sum}) using similar arguments as above.
Indeed, let $A=P\cap \G$, and suppose that  $|A| \ge s^{1-\d/4}$.
Here $P$ as before, $|P| = s$.
Thus $\T^\m_k (A) \ge |A|^{2k}/|\G|$ and we obtain
$$
    \log |A|
        \le
            \frac{1}{2} \log s
                +
                    \frac{\log s}{2\log p} \log |\G|
                        +
                            \frac{\log p}{2\log s} \log \left( \frac{C\log^2 s}{\log p} \right)
                                +
                                    \frac{\log^2 s}{\log^2 p} \log |\G|  \,.
$$
Hence by $|A| \ge s^{1-\d/4}$ and $|\G| = p^{1-\d}$, we have
$$
    \frac{\d}{4} \log s
        \le
            \frac{\log p}{2\log s} \log \left( \frac{C\log^2 s}{\log p} \right)
                +
                    \frac{\log^2 s}{\log^2 p} \log |\G|
                        \ll
                            \frac{\log p}{\log s} \log \left( \frac{C'\log^2 s}{\log p} \right) \,,
$$
where $C'>0$ is another absolute constant.
In other words $x \ll \d^{-1} \log C'x$ as above.
This completes the proof.
\end{proof}
$\hfill\Box$

Thus, the statement above is nontrivial if $|\G| \ll p/(\log p)^{C_1}$,
where $C_1>0$ is
a sufficiently large
constant.
Using Theorem  \ref{t:CSZ} one can obtain a similar result for arithmetic progressions of general form.

Further results on arithmetic progressions in subgroups can be found in \cite{BKS_oracle}.

\section{Applications : general sets}
\label{sec:applications2}

Now we find applications of Proposition \ref{p:triangles_in_psi} to
some further families of sets.
Let us begin with the convex subsets of $\R$.

\begin{theorem}
    Let $A \subseteq \R$ be a convex set.
    Then
    \begin{equation}\label{f:convex_energy}
        \E (A) \ll |A|^{\frac{89}{36}} \log^{\frac{1}{2}} |A| \,.
    \end{equation}
\label{t:convex_energy}
\end{theorem}
\begin{proof}
Let $\E = \E(A)$, $\E_3 = \E_3 (A)$.
In view of Lemma \ref{l:E_3_convex}, as in the proof of Theorem \ref{t:subgroup_energy}, we have
\begin{equation}\label{f:energy_start'}
    \sum_{\a \neq 0, \beta \neq 0, \a \neq \beta ~:~ \psi (\a), \psi (\beta), \psi (\a-\beta) \gg d}\,
            \psi^2 (\a) \psi^2 (\beta) \psi^2 (\a-\beta)
            \gg \frac{\E^6}{|A|^6 \E_3}\,.
\end{equation}
where $\psi = (A\c A) (x)$ and $d = 2^{-3} \E^2 |A|^{-3} \E^{-1/2}_3$.
The last inequality implies an analog of (\ref{tmp:19.07.2012_2}), i.e.
\begin{equation}\label{tmp:19.07.2012_2'}
    d^6 \cdot \sum_{i,j,k = 1}^l 2^{2i+2j+2k} \sum_\a S_i (\a) (S_j * S_k) (\a)
            \gg \frac{\E^6}{|A|^6 \E_3}\,.
\end{equation}
One can suppose that the summation in the last formula is taken over $i \le j \le k$.
Applying Lemma \ref{l:E_3_convex}, we have
$$
    \sum_\a S_i (\a) (S_j * S_k) (\a)
        \le
            d^{-1} 2^{-i} \sum_\a (A\c A) (\a) (S_j * S_k) (\a)
                \le
                    d^{-1} 2^{-i} \E^{1/2} (S_j,A) \E^{1/2} (S_k,A)
        \ll
$$
\begin{equation}\label{tmp:E(A,B)_convex}
        \ll |A| d^{-1} 2^{-i} |S_j|^{3/4} |S_k|^{3/4} \,.
\end{equation}
By formula (\ref{f:E_3_gen_2-}) of  Theorem \ref{t:convex_ikrt} with $k=2$, we obtain $|S_i| \ll |A|^3 / (d^3 2^{3i})$.
Combining the last bound with (\ref{tmp:E(A,B)_convex}) and (\ref{tmp:19.07.2012_2'}), we get
\begin{equation}\label{}
    \frac{\E^6}{|A|^6 \E_3}
        \ll
    d^5 |A| \cdot \sum_{i,j,k = 1}^l 2^{i-j/4-k/4} |A|^{9/2} d^{-9/2}
        \ll
            d^{1/2} |A|^{11/2} 2^{l/2} \log^2 |A|
            \,.
\end{equation}
Finally, by Andrews' inequality $2^l \ll |A|^{2/3} d^{-1}$.
Using Lemma \ref{l:E_3_convex} once more
after some calculations we obtain the result.
This completes the proof.
$\hfill\Box$
\end{proof}

\begin{corollary}
    Let $A\subseteq \Z$ be a convex set and
    $$
        P_A (\theta) = \sum_{a\in A} e^{2\pi i a \theta} \,.
    $$
    Then
    $$
        \int_0^{2\pi} |P_A (\theta)|^4\, d \theta \ll |A|^{\frac{89}{36}} \log^{\frac{1}{2}} |A| \,.
    $$
\end{corollary}

\begin{remark}
    It can be appear that the argument from the proof of Theorem \ref{t:convex_energy},
    namely, an application of an upper bound $(A\c A) (x) \ll |A|^{2/3}$, $x\neq 0$
    is quite rough.
    Nevertheless it is optimal modulo our current knowledge of convex sets.
    Indeed, let $i=j=k=l$ in formula (\ref{tmp:19.07.2012_2'}).
    By Theorem \ref{t:convex_ikrt}, we just know that $|S_i|, |S_j|, |S_k| \ll |A|$.
    Further to estimate the sum $\sum_\a S_i (\a) (S_j * S_k) (\a)$
    the only one can apply
    is estimate (\ref{tmp:E(A,B)_convex}).
    Substituting all bounds in (\ref{tmp:19.07.2012_2'}), we obtain exactly (\ref{f:convex_energy}).
\end{remark}

Using
  Theorem \ref{t:convex_ikrt}
instead of Theorem \ref{t:3_d_moment}
and
apply
the arguments from
the proof of Theorem \ref{t:subgroup_energy}
one can obtain new upper bounds for $\T_k(A)$
in the case of convex $A$. We do not make such calculations.
As in the
situation
of multiplicative subgroups
using the
weighted Szemer\'{e}di--Trotter theorem would
provide better bounds, probably.

\bigskip

Now we formulate a general result concerning the  additive energy of sets with small multiplicative doubling.

\begin{theorem}
    Let $A \subseteq \R$ be a set, and $\eps \in [0,1)$ be a real number.
    Suppose that $|AA|=M|A|$, $M\ge 1$, and
    \begin{equation}\label{f:S_eps}
        |\{ x \neq 0 ~:~ (A\c A) (x) \ge |A|^{1-\eps} \}|
            \ll
                (M \log M)^{\frac{5}{3}} |A|^{\frac{1}{6} - \frac{\eps}{4}} \log^{\frac{5}{6}} |A| \,.
    \end{equation}
    Then
\begin{equation}\label{f:energy_gen}
    \E(A) \ll 
                            M \log M
                            |A|^{\frac{5}{2} - \frac{\eps}{12}} \log^{\frac{1}{2}} |A|
                \,.
\end{equation}
\label{t:energy_gen}
\end{theorem}
\begin{proof}
By Lemma \ref{l:arranging_gen}, we have $\E_3 (A) \ll M^2 \log^2 M \cdot |A|^3 \log |A|$.
Thus $\E_3 (A)$ is small for small $M$ and we can apply the arguments from the proofs of
Theorems \ref{t:subgroup_energy}, \ref{t:convex_energy}.
Using the second bound from Lemma \ref{l:arranging_gen}, and a consequence of the first estimate,
namely, $|S_i| \ll (M\log M)^2 |A|^3 /(d^3 2^{3i})$, we obtain the required bound (\ref{f:energy_gen}).
We just need to check two inequalities.
The first is that all three terms which appeared
in the cases $\a=0$, $\beta=0$, and $\a-\beta = 0$ (see the arguments from formula (\ref{tmp:19.07.2012_1})),
namely
$$
     (M \log M)^{\frac{2}{3}} |A|^{\frac{7}{3}} \log^{\frac{1}{3}} |A|
$$
are less than our upper bound (\ref{f:energy_gen}).
One can easily assure that this is the case.
The second inequality is that the sum over nonzero $x$ such that $(A\c A) (x) \ge |A|^{1-\eps}$ is small.
Denote by $S_\eps$ the set from (\ref{f:S_eps}).
If
$$
    \frac{\E^3 (A)}{|A|^3} \ll
        \sum_{\a\in S_\eps} \sum_\beta (A\c A) (\a) (A\c A) (\beta) (A\c A) (\a-\beta) \Cf_3 (A) (\a,\beta)
            \le
$$
$$
    \le
        |A| \sum_\a \sum_\beta \sum_z S_\eps (\a) (A\c A) (\beta) (A\c A) (\a-\beta) A(z) A(z+\beta)
            \le
$$
$$
            \le
                |A| \sum_\beta (S_\eps * (A\c A)) (\beta) (A\c A)^2 (\beta)
                    \le
                       |A| \E^{2/3}_3  (A) \left( \sum_\beta (S_\eps * (A\c A))^3 (\beta) \right)^{1/3}
                            \le
$$
$$
                            \le
                                |A| (M^2 \log^2 M \cdot |A|^3 \log |A|)^{2/3} |S_\eps| |A|^{4/3}
$$
then (\ref{f:energy_gen}) holds.
This completes the proof.
$\hfill\Box$
\end{proof}

The result with $|A|^{\frac{5}{2}}$
instead of $|A|^{\frac{5}{2} - \frac{\eps}{12}}$ was known before (see \cite{ss_E_k}).

Clearly, Theorem \ref{t:energy_gen} implies Theorem \ref{t:convex_energy}, because for $\eps = 1/3$
the set from (\ref{f:S_eps}) is empty by Andrews result.
Note also that upper bound (\ref{f:S_eps}) is quite rough
and just shows
the
main idea.

\bigskip

Apply Theorem \ref{t:energy_gen} for a new family of sets $A$ with small quantity $|A(A+1)|$.
Such sets were considered in \cite{A(A+1)}, where
the following lemma was proved.

\begin{lemma}
    Let $A,B\subseteq \R$ be two sets, and $\tau \le |A|,|B|$ be a parameter.
    Then
\begin{equation}\label{}
    | \{ s\in AB ~:~ |A\cap sB^{-1}| \ge \tau \} | \ll \frac{|A(A+1)|^2 |B|^2}{|A| \tau^3} \,.
\end{equation}
\label{l:arranging_product}
\end{lemma}

Lemma above implies that for any $A\subseteq \R$
the following holds
$\E^\m (A) \ll |A(A+1)| |A|^{3/2}$.
We obtain
better
upper bound for $\E^\m (A)$
(see inequality (\ref{f:E^m_A(A+1)}) of Corollary \ref{c:f:E^m_A(A+1)} below).
Also in \cite{A(A+1)} a series of interesting inequalities were obtained.
Here we formulate just one result.

\begin{theorem}
    Let $A\subseteq \R$ be a set.
    Then
    $$
        \E^\m (A,A(A+1)) \,, ~ \E^\m (A+1,A(A+1)) \ll |A(A+1)|^{5/2} \,.
    $$
\end{theorem}

We prove the following

\begin{corollary}
    Let $A\subseteq \R$ be a set, $a\in \R$ be a number, $|A(A+1)| = M|A|$, $M\ge 1$,
    and inequality (\ref{f:S_eps}) holds in multiplicative form.
    Then
    \begin{equation}\label{f:E^m_A(A+1)-}
        \E^\m (A,A+a) \ll
                            M |A|^{\frac{5}{2} - \frac{\eps}{12}} \log^{\frac{1}{2}} |A|
    \end{equation}
    In particular
    \begin{equation}\label{f:E^m_A(A+1)}
        \E^\m (A) \ll
                        M |A|^{\frac{5}{2} - \frac{\eps}{12}} \log^{\frac{1}{2}} |A|
    \end{equation}
\label{c:f:E^m_A(A+1)}
\end{corollary}
\begin{proof}
Put $A'=A+a$ and
$
    \psi (x) = |\{ a_1,a_2 \in A ~:~ x= a_1 a^{-1}_2 \} |\,.
$
Then as in (\ref{tmp:19.07.2012}), we have
$$
    \left( \frac{\E^\m (A',A)}{|A|} \right)^3
        \le
            \sum_{\a,\beta} \psi (\a) \psi (\beta) \psi (\a \beta^{-1}) \Cf_3 (A') (\a,\beta) \,.
$$
Lemma \ref{l:arranging_product} implies that $\E^\m_3 (A') \ll M^2 |A|^3 \log |A|$.
After that apply the arguments from the proof of Theorem \ref{t:energy_gen}.
$\hfill\Box$
\end{proof}

\bigskip

Previous results of the section say, basically, that if $\E_3 (A)$ is small
and
$A$ has some additional properties such as condition (\ref{f:S_eps}) from Theorem \ref{t:energy_gen}
(which shows that $A$ is "unstructured"\, in some sense)
then
we can say something nontrivial about the additive energy of $A$.
Now we formulate (see Theorem \ref{t:E_3_M})
a variant of the principle using just smallness of $\E_3 (A)$
to show that $A$ has a structured subset.
The first result of the type was proved in \cite{ss_E_k} (see Theorem 23).

\begin{theorem}
\label{t:small E_3}
Let $A$ be a subset of an abelian group.
Suppose that $|A-A|= K|A|$ and $\E_{3}(A)=M^{}|A|^{4}/K^{2}.$ Then
there exists $A'\subseteq A$ such that $|A'|\gg |A|/M^{5/2}$
 and
 $$|nA'-mA'|\ll M^{12(n+m)+5/2}K|A'|$$
 for every $n,m\in \N.$
\end{theorem}

One can see that Theorem \ref{t:small E_3} has a strong condition,  namely,
the cardinality of the set $A-A$
is small.
Theorem \ref{t:E_BSzG2}
below
was proved in \cite{ss_E_k} (see Theorem 53, section 9)
and do not assume  any restrictions on doubling constants
but require
a stronger condition for the higher  moment, namely,
$\E_{3+\eps} (A) = M|A|^{4+\eps}/K^{2+\eps}$, $\eps \in (0,1]$.

\begin{theorem}\label{t:E_BSzG2}
Let $A\subseteq \Gr$ be a set.
Suppose that $\E(A)= |A|^3/K$ and $\E_{3+\eps}(A)= M|A|^{4+\eps}/K^{2+\eps},$
where $\eps \in (0,1]$.
Then there exists $A'\subseteq A$ such that $|A'|\gg M^{-\frac{3+6\eps}{\eps (1+\eps)}} |A|$ and
$$|nA'-mA'|\ll M^{6(n+m)\frac{3+4\eps}{\eps(1+\eps)}}
K |A'|$$ for every $n,m\in \N.$
\end{theorem}

Note that if $\eps \to 0$ then the bounds in Theorem \ref{t:E_BSzG2} becomes very bad.
Finally we formulate Theorem 51 from \cite{ss_E_k}, where the condition on the higher moment is relaxed
but the obtained bound on the doubling constant is not so good.

\begin{theorem}\label{t:BSzG1}
Let $A$ be a subset of an abelian group. Suppose that $
\E(A)=|A|^3/K$ and $\E_{2+\eps}(A)=M|A|^{3+\eps}/K^{1+\eps}\,.$ Then there
exists $A'\subseteq A$ such that $|A'|\gg |A|/(2M)^{1/\eps}$ and
$$|A'-A'|\ll 2^{\frac{6}{\eps}} M^{\frac{6}{\eps}} K^4|A'|\,.$$
\end{theorem}

Let us formulate our result.

\begin{theorem}
    Let $A\subseteq \Gr$ be a set, $\E(A) = |A|^3/K$, and $\E_3 (A) = M|A|^4 / K^2$.
    Suppose that $M\le |A| / (6K)$.
    Then there is a real number $r$
    \begin{equation}\label{cond:r}
        1\le r \le \frac{1}{|A|} \max_{x\neq 0} (A\c A)(x) \cdot K M^{1/2} \le K M^{1/2} \,,
    \end{equation}
    and a set $A' \subseteq A$
    such that
    \begin{equation}\label{f:E_3_size}
        |A'| \gg  M^{-23/2} r^{-2} \log^{-9} |A| \cdot |A|  \,,
    \end{equation}
    and
    \begin{equation}\label{f:E_3_doubling}
        |nA'-mA'| \ll (M^9 \log^{6} |A|)^{7(n+m)} r^{-1} M^{1/2} K |A'|
    \end{equation}
    for every $n,m\in \N$.
\label{t:E_3_M}
\end{theorem}
\begin{proof}
Let $\E = \E(A)$, $\E_3 = \E_3 (A)$, $\psi = A\c A$.
Then as in (\ref{tmp:19.07.2012}), we have
$$
    \left( \frac{\E (A)}{|A|} \right)^3
        \le
            \sum_{\a,\beta} \psi (\a) \psi (\beta) \psi (\a - \beta) \Cf_3 (A) (\a,\beta) \,.
$$
Using the assumption $M\le |A| / (6K)$, we get
$$
    2^{-1} \left( \frac{\E (A)}{|A|} \right)^3
        \le
            \sum_{\a \neq 0,\beta \neq 0, \a \neq \beta} \psi (\a) \psi (\beta) \psi (\a - \beta) \Cf_3 (A) (\a,\beta) \,.
$$
As before
\begin{equation}\label{f:energy_start''}
    \sum_{\a \neq 0, \beta \neq 0, \a \neq \beta ~:~ \psi (\a), \psi (\beta), \psi (\a-\beta) \gg d}\,
            \psi^2 (\a) \psi^2 (\beta) \psi^2 (\a-\beta)
            \gg \frac{\E^6}{|A|^6 \E_3}\,,
\end{equation}
where  $d = 2^{-3} \E^2 |A|^{-3} \E^{-1/2}_3$.
In terms of the sets $S_i$, we obtain a variant of formula (\ref{tmp:19.07.2012_2}), namely
\begin{equation}\label{tmp:19.07.2012_2''}
    d^4 \cdot \sum_{j,k = 1}^l 2^{2j+2k} \sum_\a (A\c A)^2 (\a) (S_j * S_k) (\a)
            \gg \frac{\E^6}{|A|^6 \E_3}\,.
\end{equation}
Trivially
$$
    |S_i| (d 2^{i-1})^3 \le \E_3 \,,
$$
and
whence
\begin{equation}\label{tmp:28.07.2012_1}
    |S_i| \ll \E_3 / (d^3 2^{3i}) \,.
\end{equation}
Note also that  $d2^i \le \max_{x\neq 0} (A\c A) (x)$, $i\in [l]$ and hence
$$
    2^i
        \le \frac{1}{|A|} \max_{x\neq 0} (A\c A)(x) \cdot K M^{1/2} \le K M^{1/2} \,.
$$
Because of
\begin{equation}\label{tmp:05.08.2012_-1}
    \sum_\a (A\c A)^2 (\a) (S_j * S_k) (\a) \le \E^{2/3}_3 \left( \sum_{\a} (S_j * S_k)^3 (\a) \right)^{1/3}
        \le
            \E^{2/3}_3 (|S_j| |S_k|)^{1/6} \E^{1/3} (S_j,S_k)
\end{equation}
then using (\ref{tmp:28.07.2012_1}), we can assume that the summation in (\ref{tmp:19.07.2012_2''})
is taken over $j,k$ such that
\begin{equation}\label{tmp:05.08.2012_1}
    \E (S_j,S_k) \gg \frac{|S_j|^{3/2} |S_k|^{3/2}}{M^9 \log^{6} |A|} := \mu |S_j|^{3/2} |S_k|^{3/2} \,.
\end{equation}
Applying (\ref{tmp:28.07.2012_1}), (\ref{tmp:05.08.2012_-1}) and a trivial upper bound for the additive energy,
namely,
$\E (S_j,S_k) \le |S_j|^{3/2} |S_k|^{3/2}$, we obtain
$$
    d^4 \cdot \sum_{j,k = 1}^l 2^{2j+2k} \sum_\a (A\c A)^2 (\a) (S_j * S_k) (\a)
        \ll
            d^4 \E^{2/3}_3 \cdot \sum_{j,k = 1}^l 2^{2j+2k} |S_j|^{2/3} |S_k|^{2/3}
                \ll\
$$
$$
                \ll
                    d^2 \E^{4/3}_3 \log^2 |A| \cdot \max_j 2^{2j} |S_j|^{2/3} \,.
$$
Thus the summation in (\ref{tmp:19.07.2012_2''}) is taken over $j\in [l]$ such that
\begin{equation}\label{tmp:05.08.2012_2}
    2^{j} |S_j| \gg 2^{-2j} M^{-2} K \log^{-3} |A| \cdot |A| \,.
\end{equation}
By Balog--Szemer\'{e}di--Gowers Theorem \ref{t:BSzG_1.5} and estimate (\ref{tmp:05.08.2012_1})
there are
$S' \subseteq S_j$, $S'' \subseteq S_k$ such that
$|S'| \gg \mu |S_j|$, $|S''| \gg \mu |S_k|$
and
$
    |S'+S''| \ll \mu^{-7} |S'|^{1/2} |S''|^{1/2}
$.
Suppose for definiteness that $|S''| \ge |S'|$.
Then
$$
    |S'+S''| \ll \mu^{-7} |S''| \,.
$$
Pl\"{u}nnecke--Ruzsa inequality (see e.g. \cite{tv}) yields
\begin{equation}\label{tmp:31.07.2012_1}
    |nS'-mS'| \ll \mu^{-7(n+m)} |S'| \,,
\end{equation}
for every $n,m \in \N$.
Using the definition of the set $S_j$ and inequality (\ref{tmp:05.08.2012_2}), we find $x\in \Gr$ such that
\begin{equation}\label{tmp:31.07.2012_2}
    |(A-x) \cap S'| \ge 2^{j-1} d |A|^{-1} |S'| \gg K^{-1} M^{-1/2} \mu 2^{j} |S_j|
        \gg
            M^{-23/2} 2^{-2j} \log^{-9} |A| \cdot |A| \,.
\end{equation}
Put $A' = A\cap (S'+x)$.
Using (\ref{tmp:31.07.2012_1}), (\ref{tmp:31.07.2012_2}) and the definition of $d$,
we obtain for all $n,m \in \N$
$$
    |nA'-mA'| \le |nS'-mS'| \ll \mu^{-7(n+m)} 2^{-j} |A| d^{-1} |A'|
        \ll
            \mu^{-7(n+m)} 2^{-j} K M^{1/2} |A'|
$$
and the result follows with $r=2^{j}$.
$\hfill\Box$
\end{proof}

Thus, for small $r$ our result is better than Theorem \ref{t:E_BSzG2} and Theorem \ref{t:BSzG1}
because we assume that just $\E_3 (A)$ is small and we obtain better bound for the doubling constant of $A'$,
correspondingly.
If $r$ is large than lower bound (\ref{f:E_3_size}) for cardinality of $A'$ is not so good
but upper bound (\ref{f:E_3_doubling}) for the doubling constant
becomes better than in Theorems \ref{t:E_BSzG2}, \ref{t:BSzG1}
as well as in Theorem \ref{t:small E_3}.

Note, finally, that condition (\ref{cond:r}) can be certainly relaxed
in spirit of assumption (\ref{f:S_eps}) from Theorem \ref{t:energy_gen}.

\bigskip

In the end of the section we give one more variant of
the arguments, using eigenvalues method.

\begin{theorem}
    Let $A\subseteq \Gr$ be a set, $D\subseteq A-A$, $D=-D$,
    $\eta \in (0,1]$ be a real number,
    \begin{equation}\label{cond:sigma_D_A}
        \sum_{x\in D} (A\c A) (x) = \eta |A|^2 \,,
    \end{equation}
    and $\E_3 (A) = \eta^3 M|A|^6 / |D|^2$.
    Then there is a set $A' \subseteq A$
    such that
    \begin{equation}\label{f:E_3_size*}
        |A'| \gg  \frac{\eta^{16} |A|}{M^5}  \,,
    \end{equation}
    and
    \begin{equation}\label{f:E_3_doubling*}
        |nA'-mA'| \ll \left( \frac{\eta^{15}}{M^5} \right)^{-7(n+m)} \frac{|D|}{\eta |A|} \cdot  |A'|
    \end{equation}
    for every $n,m\in \N$.
\label{t:E_3_sigma_D_A}
\end{theorem}
\begin{proof}
    Let $\E=\E(A)$, $\psi = A\c A$, $\sigma$ be the sum from (\ref{cond:sigma_D_A}),
    and
    $$
        D_* = \{ x\in D ~:~ (A\c A) (x) \ge 2^{-1} \eta |A|^2 / |D| \} \,.
    $$
    Clearly, $D_*=-D_*$.
    Put
    \begin{equation}\label{tmp:22.08.2012_3-}
        \sigma_* = \sum_{x\in D_*} (A\c A) (x) \ge 2^{-1} \sigma = 2^{-1} \eta |A|^2 \,.
    \end{equation}
    Denote by $\{ f_j \}_{j\in [|A|]}$ the orthonormal eigenfunctions
    of the symmetric operator $\ov{\oT}^{N^{-1}\FF{D_*}}_A$.
    Of course $f_0 \ge 0$.
    As in Proposition \ref{p:triangles_in_psi} and as in formula (\ref{tmp:19.07.2012}), we get
$$
    \sum_{\a,\beta} D_* (\a) D_* (\beta) (A\c A) (\a-\beta) \Cf_3 (A) (\a,\beta)
        =
            \sum_{x,y,z\in A} D_* (x-y) \ov{D_* (x-z)} (A\c A) (y-z)
                    =
$$
\begin{equation}\label{tmp:22.08.2012_4}
                    =
                \sum_{j} |\mu_j (\ov{\oT}^{N^{-1}\FF{D_*}}_A) |^2 \cdot \langle \ov{\oT}^{N^{-1}\FF{\psi}}_A f_j, f_j \rangle \,.
\end{equation}
Because of $\FF{\psi} \ge 0$, we obtain
$$
    \omega_j := \langle \ov{\oT}^{N^{-1}\FF{\psi}}_A f_j, f_j \rangle = \sum_x \psi(x) (f_j \c \ov{f_j}) (x)
        \ge 0 \,, \quad \quad j\in [|A|] \,.
$$
Trivially
\begin{equation}\label{tmp:23.08.2012_1}
    \mu_0 := \mu_0 (\ov{\oT}^{N^{-1}\FF{D_*}}_A) \ge |A|^{-1} \sigma_* \,.
\end{equation}
Let us estimate $\omega_0$.
We have
$$
    \mu_0 f_0 (x) = A(x) (D_* * f_0) (x) \,.
$$
By Cauchy--Schwarz, we get
$$
    \mu^2_0 \left( \sum_x f_0 (x) \right)^2 \le |D_*| \sum_x (f_0 \c A)^2 (x) = |D_*| \omega_0 \,.
$$
Using estimate (\ref{f:g_bound}) of Lemma \ref{l:g_bound} and the formula above, we obtain
\begin{equation}\label{tmp:23.08.2012_2}
    \omega_0 \ge \frac{\mu^3_0}{|D_*|} \,.
\end{equation}
Applying
(\ref{tmp:23.08.2012_1}), (\ref{tmp:23.08.2012_2}), we get
$$
    \sum_{\a,\beta} D_* (\a) D_* (\beta) (A\c A) (\a-\beta) \Cf_3 (A) (\a,\beta)
        \gg
            \frac{\mu^5_0}{|D_*|} \gg \frac{\sigma^5_*}{|A|^5 |D_*|}\,.
$$
Using the upper bound for $\E_3 (A)$ and estimate (\ref{tmp:22.08.2012_3-}), we have
$$
     \sum_x (D_* \c D_*) (x) (A \c A)^2 (x) \gg \frac{\eta^3 \sigma^4_*}{M |A|^4} \,.
$$
Applying the arguments from (\ref{tmp:05.08.2012_-1}), we get
$$
    \E(D_*) \gg \frac{\eta^{15} |D|^3}{M^5}
                    = \nu |D_*|^3 \,.
$$
By Balog--Szemer\'{e}di--Gowers Theorem \ref{t:BSzG_1.5}
there is $D' \subseteq D_*$, $|D'| \gg \nu |D_*|$ such that
$
    |D'+D'| \ll \nu^{-7} |D'|
$.
Pl\"{u}nnecke--Ruzsa inequality (see e.g. \cite{tv}) yields
\begin{equation}\label{tmp:31.07.2012_1*}
    |nD'-mD'| \ll \nu^{-7(n+m)} |D'| \,,
\end{equation}
for every $n,m \in \N$.
Using the definition of the set $D_*$ and the number $\nu$,
we find $x\in \Gr$ such that
\begin{equation}\label{tmp:31.07.2012_2*}
    |(A-x) \cap D'|
        \ge 2^{-1} \eta |A| |D'| / |D|
            \ge
                2^{-1} \eta |A| \nu |D_*| / |D|
        \gg
            \frac{\eta^{16} |A|}{M^5} \,.
\end{equation}
Put $A' = A\cap (D'+x)$.
Using (\ref{tmp:31.07.2012_1*}), (\ref{tmp:31.07.2012_2*}),
we obtain for all $n,m \in \N$
$$
    |nA'-mA'| \le |nD'-mD'| \ll \nu^{-7(n+m)} |D'|
        \ll
            \nu^{-7(n+m)} \eta^{-1} |D| |A|^{-1} |A'|
$$
and the result follows.
$\hfill\Box$
\end{proof}

Taking $D=A-A$ in
Theorem \ref{t:E_3_sigma_D_A}, we
obtain Theorem \ref{t:small E_3} (with a little bit different constants).
Thus the result above is a generalization of Theorem \ref{t:small E_3}.

\bigskip

\no{Division of Algebra and Number Theory,\\ Steklov Mathematical
Institute,\\
ul. Gubkina, 8, Moscow, Russia, 119991}
\\
and
\\
Delone Laboratory of Discrete and Computational Geometry,\\
Yaroslavl State University,\\
Sovetskaya str. 14, Yaroslavl, Russia, 150000
\\
and
\\
IITP RAS,  \\
Bolshoy Karetny per. 19, Moscow, Russia, 127994\\
{\tt ilya.shkredov@gmail.com}

\end{document}